\documentclass[11pt]{amsart}
\emergencystretch=\maxdimen
\usepackage{cite,array}  
\usepackage[all]{xy}     
\usepackage{amssymb, amsmath,amsthm,color,amsfonts, mathrsfs}
\usepackage{subfiles}  
\usepackage[CJKbookmarks=true]{hyperref}
\usepackage[capitalize,nameinlink,noabbrev,nosort]{cleveref}

\newtheorem{thm}{Theorem}[section]
\newtheorem{prop}[thm]{Proposition}
\newtheorem{lemma}[thm]{Lemma}

\newtheorem{defn}[thm]{Definition}
\newtheorem{example}[thm]{Example}

\newtheorem{remark}[thm]{Remark}

\newtheorem{qus}[thm]{Question}

 \numberwithin{equation}{section}

\newcommand{\dbb}[1]{[\![#1]\!]}

\newcommand{\bbC}{\mathbb C}
\newcommand{\bbP}{\mathbb P}

\newcommand{\bbZ}{\mathbb Z}

\begin{document}

\title{Quantum Schubert calculus for smooth Schubert divisors of $F\ell_n$ }

\author{Changzheng Li}
 \address{School of Mathematics, Sun Yat-sen University, Guangzhou 510275, P.R. China}
\email{lichangzh@mail.sysu.edu.cn}

\author{Jiayu Song}
 \address{School of Mathematics, Sun Yat-sen University, Guangzhou 510275, P.R. China}
\email{songjy29@mail2.sysu.edu.cn}

\author{Rui Xiong}
\address{Department of Mathematics and Statistics, University of Ottawa, 150 Louis-Pasteur, Ottawa, ON, K1N 6N5, Canada}
\email{rxion043@uottawa.ca}

\author{Mingzhi Yang}
 \address{Department of Mathematics, The University of Hong Kong, Hong Kong, P.R. China}
\email{yangmzh@hku.hk}

\thanks{}

\maketitle

\begin{abstract}
  We propose to study the quantum Schubert calculus for Schubert varieties, and investigate the smooth Schubert divisors   $X$ of the complete flag variety $F\ell_n$. We provide a Borel-type ring presentation of the  quantum cohomology of    $X$. We derive the quantum Monk-Chevalley formula for $X$ by   geometric arguments. We also show that the quantum Schubert polynomials for $X$ are  the same   as that for $F\ell_n$ introduced by Fomin, Gelfand and Postnikov.
\end{abstract}

\section{Introduction}

In a forthcoming paper \cite{LRY},  the first and fourth authors, together with  Rietsch, initiate a Peterson program for all Schubert subvarieties $X_{w, P}$ of flag varieties $G/P$,
including   mirror symmetry for Schubert varieties as one remarkable ingredient.
They construct a uniform Landau-Ginzburg model mirror to $X_{w, P}$, generalizing the constructions in  \cite{Rie08, RiWi}, and conjecture that
its Jacobi ring is  isomorphic to the small quantum cohomology ring of $X_{w, P}$  whenever $X_{w, P}$ is smooth and  Fano. This strongly motivates   the study of   quantum cohomology for  smooth Schubert varieties. A natural starting point is provided by   smooth Schubert varieties in the complete flag variety
      $F\ell_n:=\{V_1 \leq  \cdots \leq V_{n-1}\leq \bbC^n\mid \dim V_i=i, \forall 1\leq i<n\}$, which  is the quotient of  $G=SL(n, \mathbb{C})$ by the Borel subgroup of  {upper triangular matrices} in $G$. Let $F_1$ be a fixed one-dimensional vector subspace of $\mathbb{C}^n$. In this paper, we focus on the Schubert divisor
      $$X=\{V_\bullet \in F\ell_{n}\mid F_1\leq V_{n-1}\},$$
and study it systematically   from the viewpoint of quantum Schubert calculus.
  Although   $X$ looks  rather special and very close to the flag variety $F\ell_n$, we emphasize that our main results for $X$   already suggest features of   what a more general theory for   smooth Fano Schubert varieties would look  like; see the discussion at the end of the introduction.

We briefly recall the framework   of Schubert calculus. Schubert problems, which count the number of   geometric objects  with  given geometric  constraints,
are fundamental to enumerative geometry. Here the central objects are flag varieties $G/P$ together with their Schubert subvarieties.
The \textit{classical Schubert calculus}, in modern language, is about the study of the integral cohomology ring $H^*(G/P, \mathbb{Z})$. The Schubert classes $\sigma^u$ of the Schubert varieties form an additive basis of $H^*(G/P, \mathbb{Z})$. A satisfactory description  of this  ring typically  consists of  the following three parts:
\begin{enumerate}
    \item A ring presentation of the form $H^*(G/P, \mathbb{Z})=\mathbb{Z}[\mathbf{x}]/I$.
    \item A (manifestly positive) formula of the Schubert constants $c_{u, v}^w$ in the cup product $\sigma^u\cup \sigma^v=\sum_{w}c_{u, v}^w \sigma^w$.
    \item A Schubert polynomial $\mathfrak{S}_u(\mathbf{x})\in \bbZ[\mathbf{x}]$ that represents the Schubert class $\sigma^u$ in the aforementioned ring presentation  $\bbZ[\mathbf{x}]/I$.
\end{enumerate}
We  refer to the very nice article \cite{BGP} and the  references therein for the progress of classical Schubert calculus with an emphasis on the case $G=SL(n, \mathbb{C})$.
With   Gromov-Witten theory  introduced in the 1990s, the classical cohomology  $H^*(G/P, \mathbb{Z})$ can be deformed to the integral (small) quantum cohomology ring $QH^*(G/P, \bbZ)=(H^*(G/P, \bbZ)\otimes \mathbb{Z}[\mathbf{q}], \star)$, by incorporating  3-pointed, genus-0 Gromov-Witten invariants. In particular, we can write $\sigma^u\star \sigma^v=\sum_{w, \mathbf{d}}N_{u,v}^{w, \mathbf{d}}\sigma^w \mathbf{q}^{\mathbf{d}}$, where $N_{u,v}^{w, \mathbf{0}}=c_{u, v}^w$. There have been extensive studies of the \textit{quantum Schubert calculus}, namely of the quantum versions of the above (1)-(3) for
$QH^*(G/P, \bbZ)$ (see e.g. the survey \cite{LeLi} and the references therein).

All the Schubert varieties, including flag varieties as special cases, have CW complex structures by Schubert cells. In return, the integral cohomology of a Schubert variety $X_w$ inside $G/P$ is torsion free, and has an additive basis of Schubert classes indexed by the Weyl group elements $u$ satisfying $u\leq w$ with respect to the Bruhat order.

\begin{qus}
   What is  the (extended)  Schubert calculus for Schubert varieties?
\end{qus}

The natural inclusion $\iota: X_w\hookrightarrow G/P$ induces a surjective ring homomorphism $\iota^*: H^*(G/P, \mathbb{Z})\to H^*(X_w, \mathbb{Z})$ with kernel $I_w=\langle\sigma^u: u\not\leq w\rangle$. Hence, the (extended) classical Schubert calculus for Schubert varieties is  trivial in the sense that all points (1)-(3) can be reduced to that for the flag varieties, at the price that the ring presentation $H^*(X_w, \bbZ)=H^*(G/P, \bbZ)/I_w$ is not good enough.  We refer to \cite{ALP, GaRe, DMR, RWY, DiYo} for the study of the ring presentation of the cohomology of Schubert varieties.

The (extended) quantum Schubert calculus for Schubert varieties is highly nontrivial. First of all, we have to restrict to the smooth ones, since there is no Gromov-Witten theory   for singular (Schubert) varieties yet. To our knowledge, there have been very few pioneering studies \cite{Pec, MiSh, HKLS} in different contexts. The  odd symplectic Grassmannian $IG(k, 2n+1)$ is  a smooth Schubert variety of the symplectic Grassmannian $IG(k, 2n+2)$, a  flag variety $G/P$ with $G=Sp(2n, \mathbb{C})$. In \cite{Pec}, Pech studied the case $k=2$, which happens to be a general hyperplane section of the complex Grassmannian $Gr(2, 2n+1)$.
She did a relatively complete quantum Schubert calculus for a non-homogeneous Schubert variety for the first time,  by providing a ring presentation, the quantum Pieri formula (a partial formula for the quantum version of point (2), see also \cite{GLLX}), and the quantum Giambelli formula (i.e. the quantum version of (3)).
In \cite{MiSh}, Mihalcea and Shifler provided the (equivariant) quantum Chevalley formula for $IG(k, 2n+1)$ by using the curve neighborhood technique \cite{BuMi}.
In \cite{HKLS}, Hu, Ke, Li and Song provided a ring presentation for the quantum cohomology of the blowup of $Gr(2, n)$ along $Gr(2, n-1)$ for the purpose of studying mirror symmetry, which happens to be a Schubert divisor in a two-step flag variety. The special case when $n=3$ is the blowup of $\bbP^2$ at a point, which has been well studied much earlier. Despite being a very natural extension from the viewpoint of Schubert calculus, the quantum Schubert calculus for smooth non-homogeneous varieties is still largely uncharted territory, with many aspects awaiting exploration.

 Let $F_\bullet$ denote the standard complete flag. Each permutation $w\in S_n$ labels a Schubert variety $X_w$ of dimension $\ell(w)$ defined by rank conditions of the form  $X_w=\{V_\bullet\mid \dim (V_i\cap F_j)\geq m(i, j, w), \forall i, j\}$.  Note that the permutation $w_0=n\cdots 21$ in one-line notation is the longest element in $S_n$, and $s_i:=(i, i+1)$, $i<n$, denote  the simple transpositions.  Note     $X=X_{w_0s_{n-1}} \cong X_{w_0s_1}$, while all the other Schubert divisors $X_{w_0s_i}$, $2\leq i\leq n-2$, are singular.

Denote the following $n\times n$ matrices
\begin{align*}
     \begin{pmatrix}
    x_1 &  q_1 & \\
     -1 & x_2 & q_2  \\
          & \ddots    &\ddots & \ddots\\
              &    &-1 &  x_{n-2}   & q_{n-2} &    \\
      &    &   & -1   &x_{n-1} & q_{n-1} \\
      &    &      &   & -1 & x_{n}
  \end{pmatrix}\qquad
      \begin{pmatrix}
     x_1 &  q_1 & \\
     -1 &  x_2 & q_2  \\
         & \ddots    &\ddots & \ddots  \\
        &    &-1 &  x_{n-2}   & q_{n-2} &  -q_{n-1}q_{n-2} \\
         &   &   &-1   & x_{n-1} & -q_{n-1} x_{n-1} \\
          &      &   &  &-1 &  x_{n}
  \end{pmatrix}
\end{align*}
by   $M_{F\ell_n}$ and $M_{X_{w_0s_{n-1}}}$ respectively. Write
 \begin{equation}\label{coefEE}
     \det(I_n+\lambda M_{F\ell_n})=\sum_{i=0}^n E_i^n\lambda^i,\qquad \quad\det(I_n+\lambda M_{X_{w_0s_{n-1}}})=\sum_{i=0}^n \hat E_i^n\lambda^i.
 \end{equation}
The coefficients $E_i^n$, $\hat E^n_i$ may be viewed as quantizations of the $i$-th elementary symmetric polynomial $e^n_i(x_1, \cdots, x_n)$. Moreover, we notice $\hat E^n_1=E^n_1$ and $\hat E^{n}_n=(x_n-q_{n-1})E^{n-1}_{n-1}$, while the difference  between $\hat E^n_i$ and $E^n_i$ is a bit involved for $1< i<n$. As shown in \cite{GiKi, CF, Kim}, there is a canonical ring isomorphism
\begin{equation*}
   \Phi_q: QH^*(F\ell_n, \bbZ)\longrightarrow \mathbb{Z}[x_1, \cdots, x_n, q_1, \cdots, q_{n-1}]/(E_1^n, \cdots, E_n^n).
\end{equation*}
Its classical limit at $\mathbf{q}=\mathbf{0}$ gives Borel's ring isomorphism
 $\Phi: H^*(F\ell_n, \bbZ)\stackrel{\sim}\rightarrow {\mathbb{Z}[x_1, \cdots, x_n]\over (e_1^n, \cdots, e_n^n)}$ \cite{Borel}, where $e_i^n=e_i^n(x)$ denotes the $i$-th elementary symmetric polynomial in variables $x_1, \cdots, x_n$.
Our first main result is a similar quantum ring presentation for $X$.
\begin{thm}[Borel-type ring presentation]\label{thm: QHX}
    There is a canonical ring isomorphism    \begin{equation*}\label{QHXpresentation}
      \Psi_q:  QH^*(X, \bbZ)\longrightarrow\mathbb{Z}[x_1, \cdots, x_n, q_1, \cdots, q_{n-1}]\left/\big(\hat E_1^n, \cdots, \hat E^n_{n-1}, E_{n-1}^{n-1}\big).\right.
    \end{equation*}
\end{thm}
\noindent In particular, we obtain the canonical ring isomorphism
$\Psi: H^*(X, \bbZ)\stackrel{\sim}\to {\mathbb{Z}[x_1, \cdots, x_n]\over (e^{n}_1, \cdots, e^n_{n-1}, e^{n-1}_{n-1})}$ obtained in \cite{GaRe, RWY} by taking the classical limit at $\mathbf{q}=\mathbf{0}$. By ``canonical" above, we mean that $x_i$ represents the first Chern class of a specific tautological line bundle (see \cref{xiFl}  for $F\ell_n$ and  \cref{xiX} for  $X$).

For any $u\in S_n$, the Schubert class $\sigma^u$ in $H^{2\ell(u)}(F\ell_n, \bbZ)$ is given by the Poincar\'e dual of the homology class $[X_{w_0u}]$ in $H_{2\ell(w_0)-2\ell(u)}(F\ell_n, \bbZ)$.
The pullback Schubert classes $\{\xi^{u}:=\iota^*(\sigma^u)\}_{u\leq w_0s_{n-1}}$ form an additive basis of $H^*(X, \bbZ)$, and the divisor classes $\xi^{s_i}$'s generate    $QH^*(X, \bbZ)$ as a $\mathbb{Z}[\mathbf{q}]$-algebra.
Whenever referring to a transposition $t_{ij}=(i, j)$, we always assume $i<j$.
We say $u\lessdot_k ut_{ij}$ (resp. $u\lessdot_k^q ut_{ij}$) in the (quantum) $k$-Bruhat order, if both $i\leq k<j$ and $\ell(ut_{ij})=\ell(u)+1$ (resp.
$\ell(ut_{ij})=\ell(u)-\ell(t_{ij})$) hold.
Our second main result is the following quantum Monk-Chevalley formula,
in analogy with the quantum Monk's formula for $QH^*(F\ell_n, \bbZ)$ \cite{FGP} or more generally the quantum Chevalley formula for $QH^*(G/P, \mathbb{Z})$ \cite{FuWo}. It is
a special case of the quantum version of point (2) for $X$ but it determines the entire quantum product because the divisor  classes generate $QH^*(X, \mathbb{Z})$.
\begin{thm}[Quantum Monk-Chevalley formula]\label{thm: QCF}
    Let $1\leq k\leq n-1$ and $u\in S_n$ with $u\leq w_0s_{n-1}$. In $QH^*(X, \bbZ)$, we have
\begin{equation*}
    \xi^{s_k}\star \xi^u=\sum\xi^{ut_{ab}}+    \sum \xi^{ut_{ab}}q_{a}\cdots q_{b-1}+   \sum \xi^{w}q_{a}\cdots q_{n-1}-\delta_{k, n-1} \xi^uq_{n-1},
\end{equation*}
 where  the first sum is  over
$u\lessdot_k ut_{ab}\leq w_0s_{n-1}$,
the second sum  is over
$u\lessdot_k^q ut_{ab}$ with $b<n$,
and   the third sum is over $(w, a)$  that satisfies $wt_{an}\lessdot_k^q w$ and $u\lessdot_{n-1}  wt_{an}\in S_n$.
\end{thm}
\noindent By \cref{lengthQCF}, the permutation $w$ in the third sum satisfies $w\leq w_0s_{n-1}$; the same is true of  $ut_{ab}$ in the second sum.

In the above expression of $\xi^{s_k}\star \xi^u$, the first two sums are obtained by truncating   the quantum product $\sigma^{s_k}\star \sigma^u$, and the third sum is about  the quantum terms involving $q_{n-1}$ and appearing in $\sigma^{s_k}\star (\sigma^u\cup \sigma^{s_{n-1}})$.
   We emphasize that the minus sign in the fourth part is necessary and has a geometric explanation, which gives rise to a key ingredient in our geometric arguments.
Intuitively, a smooth curve with two marked points of degree $d$ with $d_{n-1}=1$ lies on $X$ if and only if the two marked points are both on $X$.
For stable maps,    an additional  correction term must be taken into consideration, due to the presence  of nodal curves.
  The fourth part may or may not be  canceled by the third part.
\begin{example}
    For the Schubert divisor $X_{w_0s_3}$ of $F\ell_4$, we have
 \begin{align*}
     \xi^{s_3s_2}\star \xi^{s_{3}}&=\xi^{s_3s_2s_3}+ 0 +(q_3\xi^{s_3s_2}+ q_2q_3)-q_3\xi^{s_3s_2}= \xi^{s_3s_2s_3}+  q_2q_3;\\
     \xi^{s_1s_3}\star \xi^{s_{3}}&=\xi^{s_2s_1s_3}+ 0+ (q_3\xi^{s_1s_2}+q_3\xi^{s_2s_1})-q_3\xi^{s_1s_3}.
 \end{align*}
\end{example}
  \noindent To show the vanishing of Gromov-Witten invariants of higher degrees,  we need   a special treatment with involved geometric arguments for  degrees  of the form $(0, \cdots, 0, 1, \cdots, 1,2)$, in addition to the use of
 the standard  curve neighborhood technique developed by Buch and Mihalcea \cite{BuMi} for Gromov-Witten invariants of other higher degrees.
 \begin{remark}
 Our proof of Theorem \ref{thm: QCF} is purely geometric, and all morphisms involved   are torus-equivariant.
  Consequently, \cref{thm: QCFbody} extends equivariantly by simply interpreting  $N_{v, u}^{w, d}$ as equivariant quantum Schubert structure constants for $F\ell_n$. Therefore,   the equivariant quantum extension of   \cref{thm: QCF} is simply obtained by   adding   into it
   the single term $(\omega_k-u(\omega_k))\xi^u$ , where $\omega_i$'s denote the fundamental weights.

 We work only in  the non-equivariant setting in the main  text,    to  keep the presentation concise   and aligned with the predictions of the Peterson program \cite{LRY}.
\end{remark}

 The Schubert polynomials $\mathfrak{S}_w\in \mathbb{Z}[x_1, \cdots, x_n]$, $w\in S_n$, were introduced by Lascoux and Sch\"utzenberger \cite{LascouxSchu}, and satisfy  $\Phi(\sigma^w)=[\mathfrak{S}_w]$ in Borel's presentation of $H^*(F\ell_n, \bbZ)$.
Moreover, every $\mathfrak{S}_w$ admits a unique linear expansion $\mathfrak{S}_w=\sum \alpha_{i_1\ldots i_{n-1}}e^1_{i_1}e^2_{i_2}\cdots e^{n-1}_{i_{n-1}}$.
In \cite{FGP}, Fomin, Gelfand and Postnikov   introduced the quantum Schubert polynomial
\begin{equation*}
    \mathfrak{S}_w^q= \sum \alpha_{i_1\ldots i_{n-1}}E^1_{i_1}E^2_{i_2}\cdots E^{n-1}_{i_{n-1}}.
\end{equation*}
They showed that $\Phi_q(\sigma^w)=[\mathfrak{S}_w^q]$ under the aforementioned canonical ring isomorphism $\Phi_q$ for  $QH^*(F\ell_n, \bbZ)$.Our third main result is the following.
Recall the ring isomorphism $\Psi_q$ in \cref{thm: QHX}.
\begin{thm}[Quantum Schubert polynomials]\label{thm: QSP} For any $w\leq w_0s_{n-1}$, we have $$\Psi_q(\xi^w)=[\mathfrak{S}_w^q].$$
\end{thm}
\noindent  The key ingredient in our proof is a  careful application of Theorem \ref{thm: QCF}, combined with a transition equation  (in \cref{InductiveQSchu}), showing  that   quantum Schubert polynomials can be completely determined by a very small part of the quantum Monk-Chevalley formula. This principle  was noticed and applied early in \cite[Theorem 4 and Remark 3.15]{LOTRZ}  in the
study of a combinatorial model for quantum double Schubert polynomials for $F\ell_n$.

On the one hand,  even though the quantum (resp. classical) Schubert polynomial of the class $\xi^w$ for $X$ is  the same as the corresponding polynomial for the class $\sigma^w$ for $F\ell_n$, the  quantum cohomology ring structures differ substantially.   As    \cref{thm: QCF} shows,  the expansion of $\mathfrak{S}_{s_i}^q \mathfrak{S}_u^q$ for $X$ is not obtained by simply truncating the formula  for $F\ell_n$. Thus the naive linear pullback $\sigma^u\mapsto \iota^*(\sigma^u) = \xi^u$ does not extend to a ring homomorphism $QH^*(F\ell_n, \bbZ)\to QH^*(X, \bbZ)$.

 On the other hand, the following modified pullback is  a ring homomorphism.
\begin{thm}[Quantum Lefschetz hyperplane principle]\label{thm: ringhom}
   There is a ring homomorphism $\iota^*_q: QH^*(F\ell_n, \bbZ)\rightarrow QH^*(X, \bbZ)$, defined by $\sigma^{s_{i}}\mapsto \xi^{s_i}$ and $q_i\mapsto q_i$ for $1\leq i\leq n-2$, and
   $$\sigma^{s_{n-1}}\mapsto \xi^{s_{n-1}}+q_{n-1},\qquad q_{n-1}\mapsto q_{n-1}\xi^{s_{n-1}}+q_{n-1}^2. $$
\end{thm}

\subsection*{Notable features of the main results}
A fundamental  insight of Dale Peterson \cite{Pet} assembles  the spectra of the quantum cohomology rings $QH^*(G/P)=QH^*(G/P, \bbZ)\otimes_{\bbZ} \bbC$, as  $P$ varies,   into  the Peterson variety in the Langlands dual complete   flag variety; see \cite{Rie03, Chow22}.
  In the forthcoming work \cite{LRY},
 a Peterson program is proposed for  Schubert varieties in  arbitrary  flag varieties.
As mentioned earlier, our main results above indicate features of a more general theory of smooth Fano Schubert varieties predicted in the Peterson program as follows.
\begin{enumerate}
    \item  Theorem  \ref{thm: QHX} suggests  that
    the quantum cohomology of a smooth Fano Schubert variety in a  partial flag variety  $F\ell_{n_1, \cdots, n_r; n}$ should admit a presentation  governed  by a matrix generalizing $M_{X_{w_0s_{n-1}}}$.The
    known ring presentation of $QH^*(F\ell_{n_1, \cdots, n_r; n})$    in   \cite{AsSa, CF2, Kim2} has   the quotient ideal generated by non-constant coefficients of a  matrix characteristic polynomial.
  Related  matrices  arise  in an  intermediate study of the Peterson program for    Schubert varieties   in    $F\ell_{n_1, \cdots, n_r; n}$, and specialize to  $M_{X_{w_0s_{n-1}}}$ in the present case.

   \item   Theorem \ref{thm: QSP} reflects a predicted functoriality  for the quantum cohomology of smooth Fano Schubert varieties. By \cite[Proposition 11.1]{Rie03},
   a Schubert class $\sigma^u_P$ in  $QH^*(F\ell_{n_1, \cdots, n_r; n})$,   viewed as a regular function on the corresponding Peterson stratum, coincides with  the \emph{restriction} of the regular function $\sigma^u_B$ on the Peterson stratum corresponding to $QH^*(F\ell_n)$. The Peterson program predicts an extension of this phenomenon  to smooth Fano Schubert varieties, and    Theorem \ref{thm: QSP} provides   evidence for this  prediction.

\end{enumerate}

Our main results may also have an impact in the wider field of genus
zero Gromov-Witten theory.
\begin{enumerate}
    \item[(3)]    Theorem \ref{thm: QCF} is very likely to play an important role in the study of  Galkin-Golyshev-Iritani's Gamma conjectures I and II \cite{GGI} in the case of $X_{w_0s_{n-1}}$, following the   strategy  in \cite{HKLS2}.

    \item[(4)]
    Through private communication with
      Galkin and Iritani,
      we learn that our Theorem \ref{thm: ringhom}  may be viewed as a special case of a new formulation of the  quantum Lefschetz hyperplane principle at the level of ring homomorphisms
      in their   forthcoming paper \cite{GaIr25}.

\end{enumerate}

In conclusion,  we would raise our arms and shout:
\begin{quotation}
\it
It is the right time for a  systematic  development of  quantum Schubert calculus for smooth Schubert varieties from   complementary perspectives!
\end{quotation}

 The paper is organized as follows. In Section 2, we introduce the necessary background.  In Section 3, we derive the ring presentation of $QH^*(X, \bbZ)$ by investigating the quantum differential equations. In Section 4, we provide the quantum Monk-Chevalley formula, by investigating the moduli space of stable maps of certain degrees  and using  the curve neighborhood technique for higher degrees.
 Finally in Section 5, we show that the quantum Schubert polynomials for $X$ coincide with  those for $F\ell_n$, by using induction based on a transition equation for  quantum Schubert polynomials.

\subsection*{Acknowledgements}
The authors would like to thank Ki Fung Chan,
Neil J.Y. Fan,  Hua-Zhong Ke, Tuong Le, Leonardo C. Mihalcea and Konstanze Rietsch  for helpful discussions. C.~Li is supported   by the National Key R \&  D Program of China No.~2023YFA1009801 and by NSFC Grant 12271529.

\section{Notations}
We review some background, and refer to \cite{BjBr, GaRe, CoKa} for more details.
\subsection{Combinatorics of $S_n$}
The  symmetric group $S_n$ is generated by simple reflections $s_i:=(i, i+1)$, $1\leq i<n$.  Every element $w\in S_n$ can be written as $w(1)\cdots w(n)$ in one-line notation, and its length    $\ell(w)$ is the cardinality of  the inversion set $\{(i, j)\mid i<j, \,\,w(i)>w(j)\}$.
The (unique) longest element in $S_n$ is given by   $w_0:=n\cdots 21$.  For any permutation $w$,  we have $\ell(w_0w)=\ell(w_0)-\ell(w)$.

Whenever referring to a transposition $t_{ab}:=(a, b)$, we   always assume $a<b$
in this paper.
 Let $u, w\in S_n$. We say  $u \lessdot w$ (i.e. $u$ is covered by $w$), if there exists  $t_{ab}$, such that $w=ut_{ab}$ and $\ell(w)=\ell(u)+1$.
  Then $u\leq w$ in the Bruhat order, if  $u$ can be transformed into $w$
  by a series of transpositions $t_{ij}$ that each increase the inversion number by   1. Note that $u\leq w$ holds if and only if for any $1\leq j\leq n$,  the increasingly  sorted list of $u(1), \cdots, u(j)$ is less than or equal to that of  $w(1), \cdots, w(j)$ in the usual partial order. In particular,
  we have
  \begin{equation}\label{wnneq1}
      u\leq w_0s_{n-1}=n\cdots 4312\quad \Longleftrightarrow \quad u(n)\neq 1.
  \end{equation}

 Note $\ell(t_{ab})=2b-2a-1$. Let $1\leq k<n$. The quantum Bruhat cover $\lessdot^q$ is defined by
 \begin{align*}
     u\lessdot^q ut_{ab}&\quad \mbox{if and only if}\quad \ell(ut_{ab})=\ell(u)-\ell(t_{ab}).
 \end{align*}
We can further define the (quantum) $k$-Bruhat cover by
\begin{align}\label{kBruhat}
    u\lessdot_k ut_{ab}&\quad \mbox{if and only if}\quad  \ell(ut_{ab})=\ell(u)+1\mbox{ and } a\leq k<b;\\
    u\lessdot_k^q ut_{ab}&\quad \mbox{if and only if}\quad \ell(ut_{ab})=\ell(u)-\ell(t_{ab})\mbox{ and } a\leq k<b.\label{quantumkBruhat}
 \end{align}

\subsection{Schubert varieties of $F\ell_n$}

Consider the complete flag  variety
$$F\ell_n:=\{V_1 \leq  \cdots \leq V_{n-1}\leq \bbC^n\mid \dim V_i=i, \forall~ 1\leq i<n\}.$$
Let $F_\bullet\in F\ell_n$ be the standard complete flag in $\bbC^n$. For   any permutation $w\in S_n$, we define the Schubert cell $X_w^\circ\subseteq F\ell_n$ (associated to $F_\bullet$) by the following rank conditions.
$$X_w^\circ:=\{V_\bullet \in F\ell_n\mid \dim(V_p\cap F_q)= \sharp \{k \in \mathbb{Z}_{>0}\mid k\leq p, w(k)\leq q\}, \forall~ 1\leq p,q\leq n\}.$$
Then $X_w^\circ \cong \bbC^{\ell(w)}$, and  $X^\circ_u\cap X^\circ_w=\emptyset$   for any $u\neq w$. The Schubert variety $X_w$ is given  by
$$X_w: =\overline{X_w^\circ}= \{V_\bullet \in F\ell_n\mid \dim(V_p\cap F_q)\geq \sharp \{k \in \mathbb{Z}_{>0}\mid k\leq p, w(k)\leq q\}, \forall~ 1\leq p,q\leq n\},$$
and  admits a cell decomposition by Schubert cells (with respect to the Bruhat order):
\begin{align}\label{Bruhat decomposition}
X_w=\bigsqcup_{u\leq w}X^\circ_u.
\end{align}
 In particular,   $F\ell_n$ is the (biggest) Schubert variety   $X_{w_0}$. The Schubert divisors are given by $X_{w_0s_i}$, $1\leq i\leq n-1$.
 Throughout  this paper, we will focus on the Schubert divisor  $X_{w_0s_{n-1}}$, whose defining rank conditions can be reduced to the single one:
 \begin{equation}
     X:=X_{w_0s_{n-1}}=\{V_\bullet \in F\ell_n\mid F_1\subset V_{n-1}\}.
 \end{equation}
Let $pr:F\ell_n \rightarrow Gr(n-1, n)$ be the natural projection  that sends $V_\bullet$ to $V_{n-1}$.
Then  the    divisor $X$ can be viewed as the zero locus of a section of the line bundle over $F\ell_{n}$,
\begin{equation} \label{vbL}
    \mathcal{L}_{\varpi_{n-1}}:=pr^{*}\mathcal{O}_{Gr(n-1,n)}(1).
\end{equation}

 We notice that  $X_{w_0s_1}\cong X_{w_0s_{n-1}}$ is smooth and isomorphic to an $F\ell_{n-1}$-bundle over $\mathbb{P}^{n-2}$, while the varieties  $X_{w_0s_i}$'s are all singular for $1<i<n-1$.

 \subsubsection{Topology of $F\ell_n$}
For $w\in S_n$, denote by $\sigma^w\in H^{2\ell(w)}(F\ell_n, \bbZ)$ the Schubert class defined by the Poincar\'e dual of the  {homology class  $[X_{w_0w}]\in H_{2\ell(w_0w)}(F\ell_n, \bbZ)$}. We have
\begin{equation*}
    (\sigma^{w_0u}, \sigma^w):=\int_{F\ell_n}\sigma^{w_0u}\cup \sigma^w=
\langle [X_u], \sigma^w\rangle =\delta_{u, w}.
\end{equation*}
\noindent Here $(\cdot, \cdot)$ denotes the Poincar\'e pairing for $F\ell_n$, and $\langle\cdot, \cdot \rangle$ (resp. $\delta_{u, w}$) denote the Kronecker pairing (resp. symbol). It follows from   \cref{Bruhat decomposition} that $H^*(F\ell_n, \bbZ)=\bigoplus_{w\in S_n}\bbZ \sigma^w$.

Denote the $i$-th elementary symmetric polynomial in variables $x_1, \cdots, x_n$ as
\begin{equation}
    e_i^n=e^n_i(x_1, \cdots, x_n).
\end{equation}
\noindent The following result is due to Borel \cite{Borel}.
\begin{prop}\label{prop:HFl}
     There is a canonical ring  isomorphism  $$\Phi: H^*(F\ell_n,\mathbb{Z})\longrightarrow \mathbb{Z}[x_1,x_2,\cdots,x_n]\left/\big(e_1^n, e_2^n, \cdots, e_n^n\big)\right..$$
 \end{prop}
\noindent  Denote $V_0=\{0\}$ and $V_n=\bbC^n$. Let $\mathcal{S}_{i}$ be the $i$-th tautological subbundle of $F\ell_{n}$, $0\leq i\leq n$, namely the fiber of $\mathcal{S}_i$ at a point $V_\bullet \in F\ell_n$ is given by the vector space $V_i$. By ``canonical" in the above proposition,  we mean \begin{equation}\label{xiFl}
\Phi^{-1}([x_i])=c_1((\mathcal{S}_{i}/\mathcal{S}_{i-1})^{\vee})\in H^{2}(F\ell_{n}, \bbZ).
\end{equation}

 \subsubsection{Topology of $X$}
  It also follows from the cell decomposition $\cref{Bruhat decomposition}$ that $H^*(X, \bbZ)$ is torsion free and has an additive basis $\{PD([X_w])\}_{w\leq w_0s_{n-1}}$.
   Let $\{\xi^w\}_{w\leq w_0s_{n-1}}$ denote the dual basis with respect to Poincar\'e pairing, namely   $(PD([X_u]), \xi^w)=
\langle [X_u], \xi^w\rangle =\delta_{u, w}$  for any $u, w\leq w_0s_{n-1}$. The natural inclusion map $\iota: X\hookrightarrow F\ell_n$ induces
a surjective ring homomorphism
\begin{equation}
  \iota^*: H^*(F\ell_n, \bbZ)\longrightarrow H^*(X, \bbZ);\quad  \iota^*(\sigma^w)=\begin{cases}
        \xi^w,&\mbox{if } w\leq w_0s_{n-1},\\
        0,&\mbox{otherwise}.
    \end{cases}
\end{equation}
We therefore call the classes $\xi^w$  the pullback Schubert classes. Moreover, the following ring presentation was obtained in \cite{GaRe, RWY}.
\begin{prop}\label{prop:HX}
     There is a canonical ring  isomorphism  $$\Psi: H^*(X,\mathbb{Z})\longrightarrow \mathbb{Z}[x_1,x_2,\cdots,x_n]\left/\big(e_1^n, e_2^n, \cdots, e_{n-1}^n, e_{n-1}^{n-1}\big)\right..$$
 \end{prop}
\noindent  Here by ``canonical", we mean \begin{equation}\label{xiX}\Psi^{-1}([x_{i}])=c_1((\iota^{*}\mathcal{S}_{i}/\iota^{*}\mathcal{S}_{i-1})^{\vee})\in H^{2}(X, \bbZ).
\end{equation}

\subsection{Quantum cohomology}
 To simplify the descriptions, we only review the necessary notions for $Y\in\{F\ell_n, X\}$.  Both $F\ell_n$ and $X$ are Fano with the first Chern class
  \begin{equation}  \label{c1Y}   c_1(Y)=\begin{cases}     2\sigma^{s_1}+2\sigma^{s_2}+\cdots+ 2\sigma^{s_{n-2}} +2\sigma^{s_{n-1}},&\mbox{if } Y=F\ell_n,\\
 2\xi^{s_1}+2\xi^{s_2}+\cdots + 2\xi^{s_{n-2}}+\xi^{s_{n-1}},&\mbox{if } Y=X.
 \end{cases}
 \end{equation}
(See \cite{LRY25} for a criterion for  a factorial  Schubert variety of general Lie type being Fano, as well as a formula of the first Chern class.) Uniformly denote by $\gamma^w$ the cohomology class $\sigma^w$ (resp.  $\xi^w$) labeled by an appropriate permutation $w$ for $Y=F\ell_n$ (resp. $X$).

Let   $\overline{\mathcal{M}}_{0,m}(Y,d)$ denote the moduli space  of
$m$-pointed, genus-$0$ stable maps
$(f: C\to Y; {\rm pt}_1, \cdots, {\rm pt}_m)$ of degree $d\in H_{2}(Y;\mathbb{Z})$.
For each  $i$, let ${ ev}_{i}:\overline{\mathcal{M}}_{0,m}(Y,d)\longrightarrow Y $  denote  the $i$-th evaluation map, which sends $(f: C\to Y; {\rm pt}_1, \cdots, {\rm pt}_m)$ to $f({\rm pt}_i)$; let  $\mathcal{L}_i$ denote the $i$-th universal cotangent line bundle on $\overline{\mathcal{M}}_{0,m}(Y,d)$.
The Gromov-Witten  invariant with  gravitational descendants, associated with nonnegative integers $a_i$ and cohomology classes $\gamma^{u_i}\in H^{*}(Y)=H^*(Y, \bbC)$,   is defined by
\begin{align*}
    \langle \tau_{a_1}\gamma^{u_1}, \tau_{a_2}\gamma^{u_2}, \cdots,  \tau_{a_m}\gamma^{u_m} \rangle_{0,m, d}^{Y}&:=\int_{[\overline{\mathcal{M}}_{0,m}(Y,d)]^{\rm vir}} \prod_{i=1}^{m}(c_1(\mathcal{L}_i)^{a_i} \cup { ev}_{i}^{*}(\gamma^{u_i})).
\end{align*}
Here   $[\overline{\mathcal{M}}_{0,m}(Y,d)]^{\rm vir}$ denotes the virtual fundamental class, which is of expected dimension $\dim Y+m-3+\int_{d}c_1(Y)$.
In particular, Gromov-Witten invariants of degree $d$   vanish  unless $d$ belongs to the Mori cone $\overline{\mbox{NE}}(Y)=\bigoplus_{j=1}^{n-1}\mathbb{Z}_{\geq 0}[X_{s_j}]\subset H_{2}(Y,\mathbb{Z})$ of effective curve classes of $Y$,   which we simply denote as $d\geq 0$.

 Let $\{\gamma_u\}_{u}\subset H^*(Y)$ be the dual basis of $\{\gamma^{u}\}_u$ with respect to the Poincar\'e pairing.
 The (small) quantum cohomology  $QH^{*}(Y)=(H^{*}(Y)\otimes \bbC[\mathbf{q}], \star)$ is a $\bbC[\mathbf{q}]$-algebra with the (small) quantum product defined by
 \begin{align} \label{quantumcoh}
  \gamma^{u}\star \gamma^{v}= \sum_{d \in \overline{\rm {NE}}(Y)} \sum_{ w }  \langle \gamma^{u},\gamma^{v},\gamma_{w}\rangle_{0, 3, d}^{Y} \gamma^{w} q^{d} ,
 \end{align}
where $q^d:=q_1^{d_1}\cdots q_{n-1}^{d_{n-1}}$ for  $d=\sum_{j=1}^{n-1}d_j[X_{s_j}] \in \overline{\mbox{NE}}(Y)$.
 The quantum variable $q_j$ is   of degree
 \begin{equation}
     \deg q_j=\int_{[X_{s_j}]}c_1(Y)=\begin{cases}
         1,&\mbox{if } Y=X \mbox{ and } j=n-1 \mbox{ both hold},\\
         2,&\mbox{otherwise}.
     \end{cases}
 \end{equation}

 Whenever $H^{*}(Y, \bbZ)\otimes \bbZ[\mathbf{q}]$  is closed under the quantum multiplication $\star$, we call it the integral quantum cohomology, simply denoted as $QH^*(Y, \bbZ)$.

\section{A Borel-type ring presentation of $QH^*(X)$ }
For $Y\in \{F\ell_n, X\}$,
we denote
 \begin{align}
     \mathcal{H}(Y):=H^{*}(Y)\otimes_{\mathbb{C}}  \mathbb{C}[\hbar]\dbb{\hbar^{-1}}\dbb{q_1, \cdots, q_{n-1}}.
 \end{align}
 In this section, we  introduce Givental's $J$-function  $J^Y(t, \hbar)$ of $Y$, viewed as an element in $\mathcal{H}(Y)$.
We then   prove the Borel-type ring presentation of $QH^*(X)$ in \cref{thm: QHX}, by finding  quantum differential equations that  annihilate $J^X(t, \hbar)$.

 At the beginning and  end of this section, we  identify  both first Chern classes in \cref{xiFl} and \cref{xiX} with $x_i$ by abuse of notation. Then  $\gamma^{s_j}=x_1+x_2+\cdots+x_j$  for   $1 \leq j \leq n-1$.
  Denoting $t=\sum_{i=1}^{n}t_ix_i \in H^{2}(Y)$, we view the quantum variables $q_j$ as functions on $H^2(Y)$, by  letting
$q_j=e^{\int_{[X_{s_j}]} t}=e^{t_j-t_{j+1}}$. Then it follows from
$\delta_{i,j}=\int_{[X_{s_i}]} \gamma^{s_{j}}$ that   $q^d=e^{\int_{d}t}=e^{\int_{d}\sum_{i}t_i(\gamma_i-\gamma_{i-1})}=e^{\int_{d}\sum_{i}(t_i-t_{i+1})\gamma_i}=q_1^{d_1}\cdots q_{n-1}^{d_{n-1}}$.

From \cref{prop:JfunX} through \cref{relationsQHX},
we    use $\iota^*x_i$ for $X$ to distinguish it from   $x_i$ for  $F\ell_n$, in order  to avoid confusion. For any   $f(x, \hbar) \in \mathcal{H}(F\ell_n)$, by  $f(\iota^*x, \hbar)$ we mean the element in $\mathcal{H}(X)$ simply obtained from   $f(x, \hbar)$ by replacing all $x_i$ with $\iota^*x_i$. In particular, by $\iota^*t$ we mean $\sum_it_i\iota^*x_i$.

\subsection{Givental's $J$-function of $X$}

The (small) quantum connection acts on the trivial $H^*(Y)$-bundle over $H^2(Y)\times \bbC^*$.     Its derivations along the $H^2(Y)$-direction  are given by
$$\nabla_{\frac{\partial}{\partial t_{i}}}:=\frac{\partial}{\partial t_{i}}+\frac{1}{\hbar}x_{i} \star,  ~~~1 \leq i\leq n,
$$
where   $\hbar$ is the coordinate of $\bbC^*$.
   This quantum connection is flat, with   the fundamental solution   $L(t,\hbar)$  to the  quantum differential equation $\nabla_{\frac{\partial}{\partial t_{i}}} ( L(t,\hbar)\alpha)=0$ given by \cite{Giv}
\begin{align*}
    L(t,\hbar)\alpha:=e^{-\frac{t}{\hbar}}\alpha+\sum_{d\neq 0 } \sum_{w}\langle \frac{e^{-\frac{t}{\hbar}} \alpha}{-\hbar-c_1(\mathcal{L}_1)} , \gamma_w  \rangle_{0,2, d}^{Y}   \gamma^w  q^{d}.
\end{align*}
\noindent Here  $\alpha \in H^{*}(Y)$,  and we take the expansion $\frac{1}{-\hbar-c_1(\mathcal{L}_1)}=\sum_{i\geq 0}(-1)^{i+1}\hbar^{-i-1}(c_1(\mathcal{L}_1))^{i}$.

\begin{defn}
      Givental's  (small) $J$-function   of $Y$
    is defined by
    \begin{align*}
        J^Y(t, \hbar)&:=L(t,\hbar)^{-1}(\mathbf{1})
        =e^{\frac{t}{\hbar}}(1+\sum_{d\neq 0}{\sum_{w}}\langle   \frac{\gamma_w}{\hbar-c_1(\mathcal{L}_1) } \rangle_{0,1, d}^{Y}  ~ \gamma^w q^{d}) ;\\
        J^Y(t_0, t, \hbar)&:=e^{t_0\mathbf{1}\over \hbar}J^Y(t, \hbar).
    \end{align*}
\end{defn}
 \noindent Here $\mathbf{1}=\gamma^{\rm id}$ is the identity element in $H^*(Y)$.

\begin{prop}[\protect{\cite[Theorem 1]{BCK}}]
       Givental's $J$-function of $F\ell_{n}$ is given by
$$J^{F\ell_n}(t,\hbar)= e^{t\over \hbar }(\sum_{d\geq 0}  q^{d} J_{d}^{F\ell_{n}}(x,\hbar) ),\qquad\mbox{where}$$
\begin{align*}
    J_{d}^{F\ell_{n}}(x,\hbar)&=\sum_{\sum_{j}d_{i,j}=d_i}\bigg(\prod_{i=1}^{n-1}\prod_{1\leq j< j^{'} \leq i} (-1)^{d_{i,j}-d_{i,j^{'}}} \frac{x_j-x_{j^{'}}+(d_{i,j}-d_{i,j^{'}})\hbar}{x_j-x_{j^{'}}}\bigg)  \notag \\
    & \qquad\qquad\bigg(\prod_{i=1}^{n-2} \prod_{1\leq j \leq i\atop 1 \leq j^{'} \leq i+1} \frac{\prod_{k=-\infty}^{0} (x_j-x_{j^{'}}+k\hbar)}{\prod_{k=-\infty}^{d_{i,j}-d_{i+1,j^{'}}} (x_j-x_{j^{'}}+k\hbar)} \bigg)\prod_{1\leq j \leq n-1} \frac{1}{\prod_{k=1}^{d_{n-1,j}} (x_j+k\hbar)^{n}} .
\end{align*}
Here the sum is over nonnegative integers  $d_{i, j}$,  with $1\leq j\leq i\leq n-1$.
\end{prop}

\begin{prop}\label{prop:JfunX}
  Givental's $J$-function of $X$  is given by
\begin{align*}
    J^{X}(\iota^{*}t,\hbar)=e^{-\frac{q_{n-1}}{\hbar}} I_{F\ell_n,X}(\iota^{*}t,\hbar), \label{mirtransform}
\end{align*}
where
$$I_{F\ell_n,X}(t,\hbar) :=e^{\frac{t}{\hbar} }\bigg(\sum_{d\geq 0} q^{d} J_{d}^{F\ell_{n}}(x,\hbar) \prod_{k=1}^{\int_{d} \sigma^{s_{n-1}}}(-x_n+k\hbar) \bigg)  $$
with   the empty product $\prod\limits_{k=1}^0$ read off as multiplication by  {$1$} by convention.
 \end{prop}
 \begin{proof}
 Expanding $I_{F\ell_n,X}(t,\hbar)$ in powers of  $\hbar^{-1}$,  we obtain $$I_{F\ell_n,X}(t,\hbar)= 1+\frac{1}{\hbar}(q_{n-1} \mathbf{1}+t)+ O(\hbar^{-2}).$$
  Recall from \cref{vbL} that the smooth Schubert divisor $X$ is realized as the zero locus of a section of the line bundle $\mathcal{L}_{\varpi_{n-1}}$  over $F\ell_{n}$ with $c_1(\mathcal{L}_{\varpi_{n-1}})=\sigma^{s_{n-1}}=-x_n$. By using the quantum Lefschetz theorem \cite[Theorem 1]{Kim2} (see also \cite[Corollary 7]{CoGi})  with respect to $F\ell_n$ and $X$\footnote{The function $I_{F\ell_{n},X}$ in  \cite[Section 9]{CoGi} is a multiple of ours by $\hbar$ with the identification $z=\hbar$.},   we obtain
  $$J^{X}(q_{n-1}\mathbf{1}+\iota^{*}t,\hbar)=I_{F\ell_n,X}(\iota^{*}t,\hbar)=e^{\frac{\iota^{*}t}{\hbar} }(\sum_{d\geq 0} q^{d} J_{d}^{F\ell_{n}}(\iota^{*}x,\hbar) \prod_{k=1}^{\int_{d} \sigma^{s_{n-1}}}(-\iota^{*}x_{n}+k\hbar) ). $$
This implies $J^{X}(\iota^{*}t,\hbar)=e^{-\frac{q_{n-1}}{\hbar}} I_{F\ell_n,X}(\iota^{*}t,\hbar).$
 \end{proof}

\subsection{Quantum differential equations of $X$}
In this subsection, we will investigate quantum differential equations of $X$, namely differential operators that annihilate $J^X(t, \hbar)$.

For any polynomial   $P(x,q)\in \bbC[x, q]$, we write
$$P(x,q)=\sum_d\sum_{i_I} A_{i_I}^{d}  x^{i_I}q^{d}$$
with $x^{i_I}:=x_1^{i_1}\cdots x_{n}^{i_n}$, $q^d=q_1^{d_1}\cdots q_{n-1}^{d_{n-1}}$  and the coefficients $A_{i_I}^d\in\mathbb{C}$.
Take the   conventions:
\begin{equation}\label{notaconv}
    \hbar{\partial\over \partial t_{[1, m]}}:=(\hbar{\partial\over \partial t_{1}}, \hbar{\partial\over \partial t_{2}}\cdots, \hbar{\partial\over \partial t_{m}}),\qquad q_{[1, m]}:=(q_1, q_2, \cdots, q_m),
\end{equation}
\noindent
where $q_j=e^{t_{j}-t_{j+1}}$ whenever it is treated as an operator.  We then denote
 $$P(\hbar\frac{\partial}{\partial t},q)=P(\hbar\frac{\partial}{\partial t_{[1, n]}},q_{[1, n-1]}):=\sum_d\sum_{i_I} A_{i_I}^{d}  (\hbar \frac{\partial}{\partial t_1})^{i_1}\cdots (\hbar \frac{\partial}{\partial t_n})^{i_n}q_1^{d_1}\cdots q_{n-1}^{d_{n-1}},$$ and call the differential operator   $P(\hbar\frac{\partial}{\partial t},q)$   \textit{the quantization of $P(x,q)$}.

Recall from \cref{coefEE} that $E^n_i=E^{n}_{i}(x,q)$ are the coefficients of the expansion of the polynomial $\det(I_n+\lambda M_{F\ell_n})=\sum_iE^n_i\lambda^i$ in variable $\lambda$, with respect to the matrix
\begin{align*}
    M_{F\ell_n}=\begin{pmatrix}
    x_1 &  q_1 & \\
     -1 & x_2 & q_2  \\
     &  -1   &x_3 & q_3\\
      &  & \ddots    &\ddots & \ddots\\
      &    &   & -1   &x_{n-1} & q_{n-1} \\
      &    &      &   & -1 & x_{n}
  \end{pmatrix}.
\end{align*}

\begin{prop}[\!\!\!\protect{\cite[Theorem 4]{Giv97}; \cite[Theorem II]{Kim3}}] \label{qhFl_n}
 Let $D_{i}^n(\hbar \frac{\partial}{\partial t},q )$ be the quantization of $E_{i}^n(x,q)$.
Then
$D_{i}^n(\hbar \frac{\partial}{\partial t},q )(J^{F\ell_n}(t,\hbar))=0$ for any $1\leq i \leq n$.
\end{prop}

 \begin{defn}\label{def: TdSd}
     Viewing $q_j=e^{t_j-t_{j+1}}$, we define the  operators
    \begin{equation*}
        \mathbb{T},\mathbb{S},\mathbb{P}_{0}: \mathcal{H}(F\ell_n) \longrightarrow \mathcal{H}(X),
    \end{equation*}
    $$  \mathbb{T}(e^{\frac{t}{\hbar}}\sum_{d \geq 0} q^{d} f_{d}(x,\hbar)):= e^{\frac{\iota^{*}t}{\hbar}}   \sum_{d \geq 0} q^{d} f_{d}(\iota^{*}x,\hbar)\prod_{k=1}^{\int_{d} \sigma^{s_{n-1}}}(-\iota^{*}x_n+k\hbar),$$
    $$\mathbb{S}(e^{\frac{t}{\hbar}}\sum_{d \geq 0} q^{d} f_{d}(x,\hbar) ):= e^{\frac{\iota^{*}t}{\hbar}} \sum_{d \geq 0} q^{d} f_{d}(\iota^{*}x,\hbar)\prod_{k=1}^{(\int_{d} \sigma^{s_{n-1}})-1}(-\iota^{*}x_n+k\hbar),$$
    $$\mathbb{P}_0(e^{\frac{t}{\hbar}}\sum_{d \geq 0} q^{d} f_{d}(x,\hbar) ):= e^{\frac{\iota^{*}t}{\hbar}} \sum_{d \geq 0,d_{n-1}=0} q^{d} f_{d}(\iota^{*}x,\hbar),$$
  where    {the products $\prod\limits_{k=1}^0$ and  $\prod\limits_{k=1}^{-1}$ are read off as multiplication by $1$.}
 \end{defn}

  \begin{lemma} \label{lemma1} As operators from  $\mathcal{H}(F\ell_n)$ to  $\mathcal{H}(X)$,  we have
  $$  [\hbar \frac{\partial}{\partial t_k}, \mathbb{T}]= [\hbar \frac{\partial}{\partial t_k}, \mathbb{S}]=[\hbar \frac{\partial}{\partial t_k}, \mathbb{P}_0]=[q_l, \mathbb{T}]=[q_l, \mathbb{S}]=[q_l, \mathbb{P}_0]=0$$ for any $1 \leq k \leq n$ and $1 \leq l \leq n-2.$ Moreover, we have
  \begin{equation}\label{operatoridentity}
      -\hbar\frac{\partial}{\partial t_n}\circ q_{n-1}\circ \mathbb{T}= \mathbb{T} \circ q_{n-1},\quad
       q_{n-1}\circ \mathbb{T}= \mathbb{S} \circ q_{n-1},\quad
        {\mathbb{T}=- \mathbb{S} \circ \hbar \frac{\partial}{\partial t_n}+(1+\iota^{*}x_n) \mathbb{P}_{0},}
  \end{equation}
as operators acting on $e^{t/\hbar } \sum_{d \geq 0} q^{d} f_{d}(x,\hbar)$ where   $f_d$ are independent of $t_i$.
  \begin{proof}
      Note that $\int_{d} \sigma^{s_{n-1}}=d_{n-1}$ and that the operators $\mathbb{T}$ and $\mathbb{S}$ are affected only by the power $d_{n-1}$ of $q_{n-1}=e^{t_{n-1}-t_{n}}$.
      Thus the identities in the first half of the statement hold.

       \begin{align*}
           q_{n-1}\circ \mathbb{T}(e^{t/\hbar } \sum_{d \geq 0} q^{d} f_{d}(x,\hbar))
           =\,&e^{\iota^{*}t/\hbar } \sum_{d=(d_1,\cdots,d_{n-1}),d_{n-1} \geq 1} q^{d} g_{d}(\iota^{*}x,\hbar)\prod_{k=1}^{d_{n-1}-1}(-\iota^{*}x_n+k\hbar) \\
           =\,&\mathbb{S} \circ q_{n-1}(e^{t/\hbar } \sum_{d \geq 0} q^{d} f_{d}(x,\hbar)),
       \end{align*}
        where $g_{d}(\iota^{*}x,\hbar)=f_{(d_1,\cdots, d_{n-2}, d_{n-1}-1)}(\iota^{*}x,\hbar)$. Thus the second identity in \cref{operatoridentity} holds.
    Denoting  $g_{d}(x,\hbar)=f_{(d_1,\cdots, d_{n-2},d_{n-1}-1)}(x,\hbar)$, we have
    \begin{align*}
           \mathbb{T} \circ q_{n-1}(e^{t/\hbar } \sum_{d \geq 0} q^d f_{d}(x,\hbar))
           =\,& \mathbb{T} (e^{t/\hbar } \sum_{d=(d_1,\cdots,d_{n-1}),d_{n-1} \geq 1} q^d g_{d}(x,\hbar))    \\
           =\,& e^{\iota^{*}t/\hbar } \sum_{d=(d_1,\cdots,d_{n-1}),d_{n-1} \geq 1} q^d g_{d}(\iota^{*}x,\hbar)\prod_{k=1}^{d_{n-1}}(-\iota^{*}x_n+k\hbar)\\
           =\,&-\hbar\frac{\partial}{\partial t_n}\circ q_{n-1}\circ \mathbb{T}(e^{t/\hbar } \sum_{d \geq 0} q^d f_{d}(x,\hbar)).
       \end{align*}
    Therefore the first identity in \cref{operatoridentity} holds as well.

    {The last identity in \cref{operatoridentity} follows directly from the definition of the operators $\mathbb{S},\mathbb{T},\mathbb{P}_0$.  Precisely,
\begin{align*}
           &-\mathbb{S} \circ \hbar \frac{\partial}{\partial t_{n}} (e^{t/\hbar } \sum_{d \geq 0} q^{d} f_{d}(x,\hbar))  \\
           =& (-\iota^{*}x_n) \cdot e^{\iota^{*}t/\hbar } \sum_{d \geq 0, d_{n-1}=0} q^{d} f_{d}(\iota^{*}x,\hbar)+ (-\iota^{*}x_n+d_{n-1} \hbar ) \cdot e^{\iota^{*}t/\hbar } \sum_{d \geq 0, d_{n-1}\geq 1} q^{d} f_{d}(\iota^{*}x,\hbar)   \\
           =&\mathbb{T}(e^{t/\hbar } \sum_{d \geq 0} q^{d} f_{d}(x,\hbar)) -(1+\iota^{*}x_n)\mathbb{P}_0( e^{t/\hbar } \sum_{d \geq 0} q^{d} f_{d}(x,\hbar)).
       \end{align*}
    }
   \end{proof}

\end{lemma}

\begin{prop} \label{J_hatrel}
      For the    quantizations $D_{i}^n(\hbar{\partial\over \partial t_{[1, n]}}, q_{[1, n-1]})$  of $E^n_i(x, q)$,
     \begin{align*}
         D_{i}^n(\hbar \frac{\partial}{\partial t_{[1,n]}}, q_{[1, n-2]},  (-\hbar \frac{\partial}{\partial t_n}) \circ q_{n-1})(I_{F\ell_n,X}(\iota^{*}t,\hbar))=0
     \end{align*}
 holds for  $1\leq i\leq n-1$.  Moreover,   we have
    \begin{align}
        (-D_{n-1}^{n-1} +  D_{n-2}^{n-2} q_{n-1})(I_{F\ell_n,X}(\iota^{*}t,\hbar))=0. \label{last_rel}
    \end{align}
\end{prop}
\begin{proof}
   It follows directly from the  definition that  $\mathbb{T}(J^{F\ell_n}(t,\hbar))=I_{F\ell_n,X}(\iota^{*}t,\hbar)$. By \cref{qhFl_n} and \cref{lemma1}, for $1\leq i\leq n-1$, we obtain
    \begin{align*}
        D_{i}^n(\hbar \frac{\partial}{\partial t_{[1,n]}}, q_{[1, n-2]},  (-\hbar \frac{\partial}{\partial t_n}) \circ q_{n-1})(\mathbb{T}(J^{F\ell_n}(t,\hbar)))
        =\,&\mathbb{T} \circ  D_{i}^n(\hbar \frac{\partial}{\partial t}, q)(J^{F\ell_n}(t,\hbar))=0.
    \end{align*}
 By \cref{lemma1}, all  $q_{l}, \frac{\partial}{\partial t_{k}}$ commute with   $\mathbb{T},\mathbb{S}, \mathbb{P}_0$    for  $ 1\leq l \leq n-2, 1\leq k \leq n-1$. It follows that   $[D_{a}^{a}, \mathbb{S}]=0$, $[D_a^a,\mathbb{T}]=0$ and $[D_a^a,\mathbb{P}_0]=0$ for $a\in\{n-1, n-2\}$. Using the identities in \cref{operatoridentity},
 we have
     \begin{align*}
       &    (-D_{n-1}^{n-1} +  D_{n-2}^{n-2} q_{n-1})(I_{F\ell_n,X}(\iota^{*}t,\hbar))\\ =\,&(-D_{n-1}^{n-1} +  D_{n-2}^{n-2} q_{n-1})(\mathbb{T}(J^{F\ell_n}(t,\hbar)))\\
   =\, &(  D_{n-1}^{n-1} \circ (\mathbb{S} \circ \hbar \frac{\partial}{\partial t_n}-(1+\iota^{*}x_n)\mathbb{P}_0 )  + D_{n-2}^{n-2} \circ \mathbb{S} \circ q_{n-1})(J^{F\ell_n}(t,\hbar)) \\
         =\,& (\mathbb{S} \circ (\hbar \frac{\partial}{\partial t_n}) \circ D_{n-1}^{n-1} + \mathbb{S}  \circ D_{n-2}^{n-2} \circ q_{n-1}-(1+\iota^{*}x_n)\mathbb{P}_0  \circ D_{n-1}^{n-1} )(J^{F\ell_n}(t,\hbar)) \\
         =\,&\mathbb{S}(D_{n}^n(J^{F\ell_n}(t,\hbar)))-(1+\iota^{*}x_n)\mathbb{P}_0  \circ D_{n-1}^{n-1}(J^{F\ell_n}(t,\hbar)) \\
         =\,&-(1+\iota^{*}x_n)\mathbb{P}_0  \circ D_{n-1}^{n-1}(J^{F\ell_n}(t,\hbar)) .
     \end{align*}
 Here the fourth equality holds by noting $D_n^n=(\hbar {\partial \over \partial t_n}) D_{n-1}^{n-1} +  D_{n-2}^{n-2} q_{n-1}$ and $D_{n}^{n}(J^{F\ell_n})=0$.
 Write $D_{n-1}^{n-1}(J^{F\ell_n}(t,\hbar))=\sum_{i \geq 0} q_{n-1}^{i} A_{i}  $, where each $A_i$ is independent of $q_{n-1}$. Then it follows from  $0=D_{n}^{n}(J^{F\ell_n})=(\hbar {\partial \over \partial t_n} \circ D_{n-1}^{n-1} +  D_{n-2}^{n-2} \circ q_{n-1})(J^{F\ell_n} )$    that $\hbar {\partial \over \partial t_n}(A_0)=x_{n}   A_0=0. $
By the projection formula, we have $\int_{X} 1_X\cup \iota^{*}(a\cup b)=\int_{F\ell_n} \iota_*(1_X) \cup (a\cup b)=\int_{F\ell_n} (-x_n) \cup  a \cup b$.
Thus if $x_n\cup a=0\in H^*(F\ell_n)$, then $\iota^*a=0\in H^*(X)$, following from the surjectivity of $\iota^*$. Conseqeuently,  we have $\mathbb{P}_0  \circ D_{n-1}^{n-1}(J^{F\ell_n}(t,\hbar))=\iota^{*} A_0=0$.
\end{proof}

 \begin{lemma} \label{J and J_hat}
  Both of the following identities hold.
    $$ (\hbar \frac{\partial}{\partial t_{n-1}}+q_{n-1})J^X(\iota^{*}t,\hbar)= e^{-\frac{q_{n-1}}{\hbar}} (\hbar \frac{\partial}{\partial t_{n-1}}) I_{F\ell_n,X}(\iota^{*}t,\hbar),$$
     $$ (\hbar \frac{\partial}{\partial t_n}-q_{n-1})J^{X}(\iota^{*}t,\hbar)= e^{-\frac{q_{n-1}}{\hbar}} (\hbar \frac{\partial}{\partial t_n}) I_{F\ell_n,X}(\iota^{*}t,\hbar).$$
\end{lemma}
\begin{proof} By  \cref{prop:JfunX},  $J^{X}(\iota^{*}t,\hbar)=e^{-\frac{q_{n-1}}{\hbar}} I_{F\ell_n,X}(\iota^{*}t,\hbar)$. Then the statement follows from direct calculations by the  Leibniz rule:
     \begin{align*}
         (\hbar \frac{\partial}{\partial t_{n-1}})(J^{X})&=(-q_{n-1})e^{-\frac{q_{n-1}}{\hbar}} I_{F\ell_n,X}+e^{-\frac{q_{n-1}}{\hbar}}(\hbar \frac{\partial}{\partial t_{n-1}})(I_{F\ell_n,X}), \\
       (\hbar \frac{\partial}{\partial t_{n}})(J^{X})&=q_{n-1}e^{-\frac{q_{n-1}}{\hbar}} I_{F\ell_n,X}+e^{-\frac{q_{n-1}}{\hbar}}(\hbar \frac{\partial}{\partial t_{n}})(I_{F\ell_n,X}).
     \end{align*}
\end{proof}

\subsection{Proofs of \cref{thm: QHX} and \cref{thm: ringhom}}
As shown in \cite[Proposition 2.2]{SiTi} by  Siebert and Tian, the quantum cohomology  $QH^*(X)$ of the Fano manifold $X$ is of the form $\bbC[x, q]/I_q$, provided that $H^*(X)=\bbC[x]/I$ and $I_q$ is generated by the corresponding quantized relations in   $QH^*(X)$ of the generators of $I$ for $H^*(X)$. Such relations can be found, by   the following  standard method due to Givental.
  \begin{prop}[\!\!\protect{\cite[Corollary 6.4]{Giv}}] \label{QDE}
  If a differential operator $P(\hbar \frac{\partial}{\partial{t_i}}, e^{t_i-t_{i+1}},\hbar)$ satisfies the equation  $ P(\hbar \frac{\partial}{\partial{t_i}}, e^{t_i-t_{i+1}},\hbar)(J^{Y}(t,\hbar))=0$, then $P(x_i,q_i,0)=0$ in  $QH^*(Y)$.
\end{prop}
We will first describe $QH^*(X)$, and then prove \cref{thm: QHX} as a consequence, by using  \cite[Proposition 11]{FuPa}, which is  a variation of \cite[Proposition 2.2]{SiTi}.

\begin{thm}\label{relationsQHX}
Let  $\chi_{j}:=\chi_{j}(\iota^{*}x,q)$ be given by $\det(I_n + \lambda \tilde M_X)=\sum_{i=0}^n\chi_i\lambda^i$ with respect to the matrix
\begin{align*}
    \tilde M_{X}=\begin{pmatrix}
    \iota^{*}x_1 &  q_1 & \\
     -1 & \iota^{*}x_2 & q_2  \\
     &  -1   &\iota^{*}x_3 & q_3\\
      &  & \ddots    &\ddots & \ddots\\
      &    &   & -1   &\iota^{*}x_{n-1}+q_{n-1} & q_{n-1}(-\iota^{*}x_{n}+q_{n-1}) \\
      &    &      &   & -1 & \iota^{*}x_{n}-q_{n-1}
  \end{pmatrix}.
\end{align*}
Then $\chi_{n}=(\iota^{*}x_n-q_{n-1})E_{n-1}^{n-1}(\iota^*x_1, \cdots, \iota^*x_{n-1}, q_1, \cdots, q_{n-2})$.
Furthermore, the quantum cohomology ring of   $X$ is canonically given by
    \begin{align*}
      QH^{*}(X)=\mathbb{C}[\iota^{*}x_1,\cdots,\iota^{*}x_n,q_1,\cdots,q_{n-1}]\left/\big(\chi_1,\chi_2,\cdots, \chi_{n-1}, {\chi_{n}\over \iota^{*}x_n-q_{n-1}}\big)\right. .
\end{align*}
\end{thm}
\begin{proof}
    By  \cref{J_hatrel} and \cref{J and J_hat}, for any $1\leq i \leq n-1$, we have
    \begin{align*}
        0=\,&e^{-\frac{q_{n-1}}{\hbar}}D_{i}^n(\hbar \frac{\partial}{\partial t_{[1,n]}}, q_{[1, n-2]}, (-\hbar \frac{\partial}{\partial t_n}) \circ q_{n-1})(I_{F\ell_n,X}) \\
         =\,& D_{i}^n(\hbar \frac{\partial}{\partial t_{[1,n-2]}},  \hbar \frac{\partial}{\partial t_{n-1}}+q_{n-1},\hbar \frac{\partial}{\partial t_n}-q_{n-1},  q_{[1, n-2]},   (-(\hbar \frac{\partial}{\partial t_n}-q_{n-1}) \circ q_{n-1})+\hbar G_1+\hbar^{2}G_2 )\bigg)(J^{X})
    \end{align*}
    where $G_1=G_1(\hbar\frac{\partial}{\partial t}, q)$ and $G_2=G_2(\hbar\frac{\partial}{\partial t}, q)$ are  differential operators. Therefore by using \cref{QDE}, we have
    $$D_{i}^n(\iota^{*}x_1,\cdots,\iota^{*}x_{n-2}, \iota^{*}x_{n-1}+q_{n-1},\iota^{*}x_n-q_{n-1} ,q_1,q_2,\cdots, q_{n-2}, q_{n-1}(-\iota^{*}x_n+q_{n-1}))=0$$ in $QH^{*}(X).$  That is, $\chi_{i}(\iota^{*}x,q)=0$ holds in $QH^*(X)$.

  The next two equalities follow directly from   \cref{J and J_hat} and  \cref{J_hatrel} respectively.
  \begin{align*}
      &\big(-D_{n-1}^{n-1}(\hbar \frac{\partial}{\partial t_{[1, n-2]}},\hbar \frac{\partial}{\partial t_{n-1}}+q_{n-1},q_{[1, n-2]})+D_{n-2}^{n-2}(\hbar \frac{\partial}{\partial t_{[1, n-2]}},q_{[1, n-2]})q_{n-1}+\hbar H_1 \big)(J^{X}) \\
      =\,&e^{-\frac{q_{n-1}}{\hbar}}\big(-D_{n-1}^{n-1}(\hbar \frac{\partial}{\partial t_{[1, n-1]}},q_{[1, n-2]})+D_{n-2}^{n-2}(\hbar \frac{\partial}{\partial t_{[1, n-2]}},q_{[1, n-2]}) q_{n-1}\big)(I_{F\ell_n,X})=0,
  \end{align*}
 where $H_1=H_1(\hbar\frac{\partial}{\partial t}, q)$ is a   differential operator. Therefore, by \cref{QDE}, in $QH^{*}(X)$  we have
 $$-E_{n-1}^{n-1}(\iota^{*}x_1,\cdots,\iota^{*}x_{n-2},\iota^{*}x_{n-1}+q_{n-1},q_1,\cdots,q_{n-2})+E_{n-2}^{n-2}(\iota^{*}x_1,\cdots,\iota^{*}x_{n-2},q_1,\cdots,q_{n-2}) q_{n-1}=0.$$
 The left hand side of the above equality equals  $-E_{n-1}^{n-1}(\iota^{*}x_1,\cdots, \iota^{*}x_{n-1},q_1,\cdots,q_{n-2})
 $ and hence equals  $-{\chi_{n}(\iota^{*}x,q)\over \iota^{*}x_{n}-q_{n-1}}$ by Laplace expansion of matrices.
  Then we are done by using  \cref{prop:HX} and    \cite[Proposition 2.2]{SiTi}.
\end{proof}

Now we are ready to show the ring presentation of the integral quantum cohomology $QH^*(X, \bbZ)$ in \textbf{\cref{thm: QHX}}, as well as the ring homomorphism in \textbf{\cref{thm: ringhom}}.

\bigskip

\begin{proof}[\textbf{Proof of \cref{thm: QHX}}]
    Abusing the notation for $\iota^*x_i$ and $x_i$, we notice
    $M_X=A\tilde M_XA^{-1}$ with $A=I_n+q_{n-1}B_{n-1, n}$, where $B_{n-1, n}$ is the matrix with 1 in the  $(n-1, n)$-entry and zeros elsewhere.
    Hence, the matrices $M_X$ and $\tilde M_X$ have the same characteristic polynomial, implying $\hat E_i^n=\chi_i$ for all $1\leq i\leq n$.
  Again we note  ${\chi_{n} \over  x_{n}-q_{n-1}}=E_{n-1}^{n-1}$. Thus all $\hat E_1^n, \cdots, \hat E_{n-1}^n, E_{n-1}^{n-1}$ vanish in $QH^*(X)=QH^*(X, \bbZ)\otimes \bbC$, so do in $QH^*(X, \bbZ)$.
  Their evaluations at $q=0$ give the ideal $I=(e_1^n, \cdots, e_{n-1}^n, e_{n-1}^{n-1})$, providing the ring presentation $H^*(X, \bbZ)=\bbZ[x_1, \cdots, x_n]/I$ by \cref{prop:HX}. Therefore the statement follows from \cite[Proposition 11]{FuPa}.
\end{proof}

\bigskip

 \begin{proof}[\textbf{Proof of \cref{thm: ringhom}}]
   Define $\iota^*_q(x_i)=x_i$ and $\iota^*_q(q_i)=q_i$ for $1\leq i\leq n-2$. Define $$\iota^*_q(x_{n-1})=x_{n-1}+q_{n-1},\qquad \iota^*_q(x_{n})=x_{n}-q_{n-1},\quad \iota^*_q(q_{n-1})=-q_{n-1}x_n+q_{n-1}^2. $$
This induces a ring homomorphism $\iota^*_q: \bbZ[x, q]\to \bbZ[x, q]$ with
 $$\iota^*_q(E_i^n(x, q))=E_i^n(x_1, \cdots, x_{n-2}, \iota^*_q(x_{n-1}), \iota^*_q(x_{n}), q_1, \cdots, q_{n-2}, \iota^*_q(q_{n-1}))=\chi_i(x, q)$$
 for all $1\leq i\leq n$.
 That is, $\iota^*_q(E_i^n(x, q))=\hat E_i^n(x, q)$ for $1\leq i\leq n-1$, and  $\iota^*_q(E_n^n(x, q))=(x_n-q_{n-1})E_{n-1}^{n-1}(x, q)$ by the proof of \cref{thm: QHX}. Hence it further induces a ring homomorphism
  $$\iota^*_q: QH^*(F\ell_n, \mathbb{Z})={\bbZ[x, q]\over (E^n_1, \cdots, E_n^n)} \longrightarrow {\bbZ[x, q]\over (\hat E_1^n, \cdots, \hat E_{n-1}^n, E_{n-1}^{n-1})}=QH^*(X, \bbZ).$$
  Hence, we are done, by noting  $\sigma^{s_i}=[x_1+\cdots+x_i]$ on the left hand side, and $\xi^{s_i}=[x_1+\cdots+x_i]$ on the right hand side. (In particular, $\xi^{s_{n-1}}=[-x_n]$.)
 \end{proof}

\section{Quantum Monk-Chevalley formula for $X$}\label{Sect-QChevalley}
Let $Y\in\{F\ell_n, X\}$. By $d$ we always mean   $d=(d_1, \cdots, d_{n-1})=\sum_{i=1}^{n-1}d_i[X_{s_i}]\in H_2(Y, \bbZ)$.
For $1\leq a<b\leq n$, we denote
\begin{equation}
    \alpha_{ab}:=\sum_{i=a}^{b-1} [X_{s_i}],\qquad  q_{ab}:= q_aq_{a+1}\cdots q_{b-1}=q^{\alpha_{ab}}.
\end{equation}
The following quantum Monk's formula for $F\ell_n$ was proved in \cite[Theorem 1.3]{FGP} (see \cite{FuWo} for the quantum Chevalley formula   for general  $G/P$), where $\lessdot_r$ (resp. $\lessdot_r^q$) denotes the (quantum) $r$-Bruhat order defined in \cref{kBruhat} (resp.  \cref{quantumkBruhat}).
\begin{prop}[Quantum Monk's formula]\label{prop: quantumMonkFl}
For $u\in S_n$ and $1\leq r<n$, in $QH^*(F\ell_n, \bbZ)$ we have
$$\sigma^{s_r}\star\sigma^u=\sum_{u\lessdot_rut_{ab}}\sigma^{ut_{ab}}+\sum_{u\lessdot_r^q ut_{ck}}q_{ck}\sigma^{ut_{ck}}.$$
\end{prop}
Since $\{\sigma^u\}_u$ form a  $\bbZ[q]$-basis of $QH^*(F\ell_n, \bbZ)$, we can write
  $$\sigma^{v}\star\sigma^{u}=\sum_{w, d}N_{v, u}^{w, d}q^d\sigma^w.$$
 We further denote by $N_{v, u\cup u'}^{w, d}$
  the coefficient of $q^d\sigma^w$ in the product $\sigma^v\star(\sigma^u\cup  \sigma^{u'})$.
 In particular if $v=s_r$ and $d\neq 0$, then  the classical  Monk's formula expresses $\sigma^u\cup\sigma^{u'}$ as a sum of  distinct $\sigma^{\tilde u}$, which contribute nonzero quantum Schubert structure constants (equal to 1) only if $d=\alpha_{ab}$ and $\tilde u=wt_{ab}$. Therefore, the quantum Monk-Chevalley formula in the form of \cref{thm: QCF} is equivalent to the following description.

\begin{thm}\label{thm: QCFbody}
    Let $1\leq r\leq n-1$ and $u\in S_n$ with $u\leq w_0s_{n-1}$. In $QH^*(X)$, we have
\begin{equation*}
   \xi^{s_r} \star \xi^{u}=\sum_{w \leq w_0s_{n-1}} N_{s_{r},u}^{w,0} \xi^{w}+\sum_{ d_{n-1}=0} N_{s_{r},u}^{w,d} \xi^{w} q^{d} +\sum_{d_{n-1}=1} N_{s_{r},(u \cup s_{n-1})}^{w,d} \xi^{w} q^{d}-\delta_{r,n-1}q_{n-1} \xi^{u}.
\end{equation*}
 \end{thm}
 \noindent The constraint $w \leq w_0s_{n-1}$ is a priori required in the second and third sums of the formula $\xi^{s_r} \star \xi^{u}$, but turns out to be redundant (see \cref{lengthQCF}).

This section is devoted to a proof of the above theorem. We use the current form in \cref{thm: QCFbody}, to  indicate our approach:
degree-$d$ Gromov-Witten invariants of $X$ with $d_{n-1}\leq 1$ can be reduced those of $F\ell_{n}$, while those  with $d_{n-1}\geq 2$ vanish.

\subsection{Degree-$d$ Gromov-Witten invariants with $d_{n-1}\leq 1$}
In this subsection, we compute Gromov-Witten  invariants $\langle \beta, \gamma\rangle_{0, 2, d}$ of $X$ with $d_{n-1}\leq 1$.
\subsubsection{Unobstructedness  of moduli spaces}
Recall the line bundle $\mathcal{L}_{\omega_{n-1}}$ over $F\ell_n$ defined in \eqref{vbL}, the zero locus of a section $s$ of which defines  the smooth Schubert divisor $X$.

 \begin{lemma} \label{no_obs}
     Let $f: \mathbb{P}^{1} \rightarrow X$ be a morphism satisfying $f_{*}[\mathbb{P}^{1}]=d$ with $d_{n-1} \leq 1$.  Then we have
     $H^{1}(\mathbb{P}^{1}, f^{*}T_{X})=0.$
 \end{lemma}

 \begin{proof}
    Set $E :=\mathcal{L}_{\omega_{n-1}}|_X \cong N_{X/F\ell_n}$.   The exact sequence $$0 \rightarrow T_{X} \rightarrow T_{F\ell_n}|_{X} \rightarrow E \rightarrow 0, $$
      pulling back to $\mathbb{P}^{1}$, and then tensoring with $\mathcal{O}_{\mathbb{P}^{1}}(-1)$, induces a long exact sequence:
      $$\cdots \rightarrow H^{0}(\mathbb{P}^{1},
      (f^{*}E)(-1)) \rightarrow H^{1}(\mathbb{P}^{1},(f^{*}T_{X})(-1)) \rightarrow H^{1}(\mathbb{P}^{1},(f^{*}T_{F\ell_n}|_{X})(-1)) \rightarrow \cdots  .$$
       Since $T_{F\ell_{n}}$ is globally generated and vector bundles over $\mathbb{P}^{1}$  split,   we have $$(f^{*}T_{F\ell_n}|_{X})(-1)=\bigoplus_{i}\mathcal{O}_{\mathbb{P}^{1}}(a_i-1)$$ with $a_i \geq 0$ for all $i$. Thus $H^{1}(\mathbb{P}^{1},(f^{*}T_{F\ell_n}|_{X})(-1))=0.$
       On the other hand, $(f^{*}E)(-1)=\mathcal{O}_{\mathbb{P}^{1}}(d_{n-1}-1)$ with $d_{n-1} \leq 1$, which implies
       \begin{align} \label{dimres}
           \text{dim} ~H^{1}(\mathbb{P}^{1},(f^{*}T_{X})(-1)) \leq 1.
       \end{align}
      Suppose $(f^{*}T_{X})(-1)=\bigoplus_{i}\mathcal{O}_{\mathbb{P}^{1}}(b_i-1)$,
      \cref{dimres} implies all $b_i \geq 0$ except at most one $b_{i_0}=-1.$ In other words, $f^{*}T_{X}=\bigoplus_{i \neq i_0}\mathcal{O}_{\mathbb{P}^{1}}(b_i) \oplus \mathcal{O}(-1)$ with $b_i \geq 0$ or $f^{*}T_{X}=\bigoplus_{i}\mathcal{O}_{\mathbb{P}^{1}}(b_i) $  with $b_i \geq 0$, which implies $H^{1}(\mathbb{P}^{1}, f^{*}T_{X})=0.$
 \end{proof}

 \begin{prop}
     Let $d\geq 0$ with $d_{n-1} \leq 1$. Then the virtual fundamental class $[\overline{\mathcal{M}}_{0,k}(X,d)]^{\rm vir}$ coincides with the  usual fundamental class  $[\overline{\mathcal{M}}_{0,k}(X,d)]$.
 \end{prop}
  \begin{proof}
   Let $ft: \overline{\mathcal{M}}_{0,k+1}(X,d) \longrightarrow  \overline{\mathcal{M}}_{0,k}(X,d) $ be the forgetful morphism obtained by forgetting the last marked point (which is the universal curve over $\overline{\mathcal{M}}_{0,k}(X,d)$  \cite{BeMa}).   For any $(f: C\to X; pt_1,\cdots, pt_k)\in \overline{\mathcal{M}}_{0,k}(X,d)$,  by using \cref{no_obs} and following (the proof of)  \cite[Lemma 10]{FuPa}, we have  $H^{1}({C},f^{*}T_{X})=0$.
   Therefore we have $R^{1}ft_{*} ev_{k+1}^{*}T_{X}=0.$
Then by \cite[Proposition 7.3]{BeFa},
    $[\overline{\mathcal{M}}_{0,k}(X,d)]^{\rm vir}$ is the usual fundamental class $[\overline{\mathcal{M}}_{0,k}(X,d)]$.
  \end{proof}

\subsubsection{Computation of Gromov-Witten invariants with $d_{n-1}= 1$}In this subsection, we always assume $d\geq 0$ with $d_{n-1}=1$. For a decomposition $d= d'+ d''$, we always require both $d'\geq0$ and $d''\geq 0$.
Recall the natural projection map $pr: F\ell_n \rightarrow Gr(n-1, n)=(\bbP^{n-1})^*$. Note that $H:=pr(X)$ is a hyperplane in $(\bbP^{n-1})^*$.
{
The  inclusion $\iota: X\hookrightarrow F\ell_n$ induces a  natural  inclusion    $\overline{\mathcal{M}}_{0,2}(X,d)  \hookrightarrow \overline{\mathcal{M}}_{0,2}(F\ell_{n}, \iota_{*}d)$ denoted as $\iota$ by abuse of notation.}
Denote
\begin{align*}
    \mathcal{A}^{\circ}&:= \big\{(f: C\to F\ell_n;\,\, pt_1,pt_2) \in  \overline{\mathcal{M}}_{0,2}(F\ell_n, \iota_*d)\mid f(pt_i)\in X, pr(f(pt_1))\neq pr(f(pt_2))\big\};\\
    \mathcal{B}&:=\big\{(f: C\to F\ell_n;\,\, pt_1,pt_2) \in  \overline{\mathcal{M}}_{0,2}(F\ell_n, \iota_*d)\mid pr(f(pt_1))= pr(f(pt_2))\in H \big\}.
\end{align*}
Since $d_{n-1}=1$, $pr\circ f(C)$ is a line in $(\bbP^{n-1})^*$, containing $pr(f(pt_1)), pr(f(pt_2))$. It follows that $f(C)\subset X$ for any  stable map $f$ in $\mathcal{A}^\circ$, because the line through  the two distinct points $pr(f(pt_1)), pr(f(pt_2))$ in $H$ must lie in  $H$. Note that $\mathcal{A}^\circ$ is a Zariski open dense subset of $\iota(\overline{\mathcal{M}}_{0,2}(X,d))$. Thus   for the evaluation maps $ev_i: \overline{\mathcal{M}}_{0,2}(F\ell_n, \iota_*d)\longrightarrow F\ell_n$, we have
\begin{equation}\label{decompd=1}
ev_1^{-1}(X)\cap ev_2^{-1}(X)
= \mathcal{A}\cup \mathcal{B},\quad\mbox{with}\quad\mathcal{A}:=\overline{\mathcal{A}^\circ}= \iota(\overline{\mathcal{M}}_{0,2}(X,d)).
\end{equation}
Here we note that $\overline{\mathcal{M}}_{0,2}(X,d)$ is proper, so $\iota(\overline{\mathcal{M}}_{0,2}(X,d))$ is closed in $\overline{\mathcal{M}}_{0,2}(F\ell_n,\iota_*d)$.

Denote by $X_{k,d}=ev_k^{-1}(X)\subset \overline{\mathcal{M}}_{0,k}(F\ell_n,\iota_*d)$ the space of stable maps whose last marked point is mapped to $X$.
A stable map $(f:C\to F\ell_n,pt_1,pt_2)$ in $\mathcal{B}$ is of the form
\begin{itemize}
    \item[(1)] $C=C_1\cup C_2$, where the marked points $pt_1,pt_2$ are contained in $C_1$;
    \item[(2)]
$f_1=f|_{C_1}$ is a stable map to $pr^{-1}(p)$ for some $p\in H$;
    \item[(3)] $f_2=f|_{C_2}$ is a stable map to $F\ell_n$.
\end{itemize}
Actually, we can take $C_1$ to be the union of components of $C$ which is maximal with respect to (1) and (2).
The   union of the remaining components $C_2$ is connected since the image $pr(f(C))$ is a line.
As a result,   $\mathcal{B}$ can be written as the following union, in analogy with the boundary divisors of the moduli space of stable maps (see e.g. \cite[Section 6.2]{FuPa}).
   \begin{equation}
       \mathcal{B}=\bigcup_{
     d'+ d''=d\atop
    d_{n-1}'=0
    }  \mathcal{B}_{d',  d''},  \quad\mbox{with}\quad\mathcal{B}_{d',d''} := X_{3,d'}\times_X X_{1,d''}.
   \end{equation}
Here the   fiber product over $X$ stands for the constraint $f_1(pt_3)=f_2(pt)\in X$ for the nodal point of $C$.
  By direct calculations, we have
  $$\dim \mathcal{B}_{ d',  d''}= \langle \iota_{*}d', c_1(F\ell_n)\rangle +\dim X + \langle \iota_{*} d'', c_1(F\ell_n)  \rangle -2= \langle \iota_{*}d, c_1(F\ell_n)\rangle +\dim X -2.$$
 That is, $\mathcal{B}$ is of pure dimension. Note $\mathcal{A}$ and $\mathcal{B}$ are both of codimension $2$ in $\overline{\mathcal{M}}_{0, 2}(F\ell_n, \iota_*d)$.

\begin{prop} \label{lowdg GW}
     Let $d\geq 0$ with $d_{n-1} =1$. For  $u, w\in S_n$ with $u, w\leq w_0s_{n-1}$,
     we have
     \begin{align*}
         \langle \xi^{u}, \xi^{w} \rangle_{0,2,d}^{X}
         =
\langle \sigma^{u}\cup \sigma^{s_{n-1}}, \sigma^{w} \cup \sigma^{s_{n-1}} \rangle_{0,2,\iota_{*}d}^{F\ell_n}- \delta_{q^{d},q_{n-1}} \langle  \sigma^{u} , \sigma^{w} , \sigma^{s_{n-1}} \rangle_{0,3,0}^{F\ell_n}.
     \end{align*}
 \end{prop}

\begin{proof}
 Note that the Schubert divisor $X$ defines the Schubert class   $\sigma^{s_{n-1}}=PD([X])$, and $ev_{i}^{-1}(X)$ are divisors of $\overline{\mathcal{M}}_{0, 2}(F\ell_n, \iota_*d)$.
 As discussed above, we have the decomposition
   $ev_1^{-1}(X) \cap ev_{2}^{-1}(X)=
 \overline{\mathcal{M}}_{0,2}(X,d)\cup \bigcup_{(d',  d'')}X_{3,d'}\times_X X_{1,d''}$ with each irreducible component   of codimension 2 in $\overline{\mathcal{M}}_{0,2}(F\ell_n, \iota_*d)$.
 Hence, we have
\begin{align*}
  &\langle \sigma^{u}\cup \sigma^{s_{n-1}}, \sigma^{w} \cup \sigma^{s_{n-1}} \rangle_{0,2,\iota_{*}d}^{F\ell_n}\\
  =&\int_{[\overline{\mathcal{M}}_{0,2}(F\ell_n,\iota_{*}d)]}   ev_1^{*} \sigma^{u} \cup {ev}_2^{*}\sigma^{w} \cup ev_1^{*}\sigma^{s_{n-1}}  \cup {ev_2}^{*}\sigma^{s_{n-1}}  \\
    =\,&\int_{[ev_1^{-1}(X) \cap ev_2^{-1}(X)]}    ev_1^{*} \sigma^{u} \cup  {ev}_2^{*}\sigma^{w}   \\
    =\,& \int_{[\overline{\mathcal{M}}_{0,2}(X,d)]}    ev_1^{*} \sigma^{u} \cup {ev}_2^{*}\sigma^{w}  +\sum_{(d',  d'')}\int_{[X_{3,d'}\times_X X_{1,d''}]}    ev_1^{*} \sigma^{u} \cup  {ev}_2^{*}\sigma^{w}.
\end{align*}
The first term in the last equality
is equal to
$$\int_{[\overline{\mathcal{M}}_{0,2}(X,d))]}\iota^*( ev_1^{*} \sigma^{u} \cup  {ev}_2^{*}\sigma^{w})
= \int_{[\overline{\mathcal{M}}_{0,2}(X,d)]} ev_1^{*} \iota^*\sigma^{u} \cup {ev}_2^{*}\iota^*\sigma^w  =\langle \xi^u, \xi^w \rangle_{0, 2, d}^X.$$
Note $ev_1\times ev_2:X_{3,d'}\times_X X_{1,d''}\to X\times X$
  factors through
$$X_{3,d'}\times_X X_{1,d''}\stackrel{\phi}{\longrightarrow}
X_{3,d'}\stackrel{ev_1\times ev_2}{\longrightarrow}
X\times X,$$
 where $\phi$ is a fibration with generic fibers of dimension the same as that of the generic fiber
 of the evaluation map $  \overline{\mathcal{M}}_{0,1}(F\ell_{n}, \iota_{*}d^{''}) \to   F\ell_n $, namely of dimension
 $\langle \iota_*d'', c_1(T_{F\ell_n})\rangle +1-3$. Hence, $\phi_*[X_{3,d'}\times_X X_{1,d''}]=0$ unless $\langle \iota_*d'', c_1(T_{F\ell_n})\rangle =2$, i.e. $d''=[X_{s_{n-1}}]=\alpha_{n-1, n}$.
 {In this case, $\overline{\mathcal{M}}_{0,1}(F\ell_n,d'')\cong F\ell_n$ (\cite{Str}), thus $X_{3,d'}\times_X X_{1,d''}\to X_{3,d'}$ is of degree $1$.}
 Hence, we have

 \begin{align*}
 \sum_{(d',  d'')}\int_{[X_{3,d'}\times_X X_{1,d''}]}    ev_1^{*} \sigma^{u} \cup  {ev}_2^{*}\sigma^{w}
    =&\sum_{(d',  d'')}\int_{[X_{3,d'}\times_X X_{1,d''}]}    \phi^*\circ (ev_1\times ev_2)^*(\sigma^{u}\boxtimes\sigma^{w}) \\ =&\sum_{(d',  \alpha_{n-1, n})}\int_{[X_{3,d'}]}    (ev_1\times ev_2)^*(\sigma^{u}\boxtimes\sigma^{w})\\
    =&\sum_{(d',  \alpha_{n-1, n})}\int_{[\overline{\mathcal{M}}_{0, 3}(F\ell_n,  \iota_*d')]}    ev_1^*\sigma^u\cup   ev_2^*\sigma^{w}\cup ev_3^*\sigma^{s_{n-1}}.
 \end{align*}
 Recall that $d'_{n-1}=0$.
 If $d'\neq 0$, then by the divisor axiom, we have  $\int_{[\overline{\mathcal{M}}_{0, 3}(F\ell_n,  \iota_*d')]}    ev_1^*\sigma^u\cup   ev_2^*\sigma^{w}\cup ev_3^*\sigma^{s_{n-1}} =
 \int_{\iota_*d'}\sigma^{s_{n-1}} \int_{\overline{\mathcal{M}}_{0,2}(F\ell_{n}, \iota_{*}d')} ev_1^{*}\sigma^{u} \cup ev_2^{*} \sigma^{w}=0$.
 Hence, the above sum is nonzero only if $d=\alpha_{n-1, n}=d''$, in which case we have
 $$ \sum_{(d',  d'')}\int_{[X_{3, d'}\times_X X_{1,  d''}]}    ev_1^{*} \sigma^{u} \cup  {ev}_2^{*}\sigma^{w}= \int_{[\overline{\mathcal{M}}_{0, 3}(F\ell_n,  0)]}    ev_1^*\sigma^u\cup   ev_2^*\sigma^{w}\cup ev_3^*\sigma^{s_{n-1}}=\langle \sigma^u, \sigma^w, \sigma^{s_{n-1}}\rangle_{0, 3, 0}^{F\ell_n}. $$
 Hence, the statement follows.
\end{proof}

\subsubsection{Computation of Gromov-Witten invariants with $d_{n-1}=0$}
The analysis for $d_{n-1}=0$ is similar to but much simpler than that for $d_{n-1}=1$.
\begin{prop} \label{lowdg00}
     Let $d\geq 0$ with $d_{n-1}=0$ and $d\neq 0$.
     For $u, w\in S_n$ with $u, w\leq w_0s_{n-1}$, we have
     \begin{align*}
         \langle \xi^{u}, \xi^{w} \rangle_{0,2,d}^{X}
         =
 \langle \sigma^{u}, \sigma^{w} \cup \sigma^{s_{n-1}} \rangle_{0,2,\iota_{*}d}^{F\ell_n}.
     \end{align*}
 \end{prop}

 \begin{proof}
   For $i\in\{1, 2\}$ and $(f:C\to F\ell_n; pt_1,pt_2) \in  ev_i^{-1}(X)\subset  \overline{\mathcal{M}}_{0,2}(F\ell_{n}, \iota_{*}d)$,
   we have $pr_*([f(C)])=0$,  {so that $pr(f(C))$ consists of a point in $\bbP^{n-1}$.}Moreover,
   $pr(f(pt_1))=pr(f(pt_2))=pr(ev_i(f)) \in H$. It follows that $f(C)\subset X$. Hence,
    $ev_i^{-1}(X)=\iota(\overline{\mathcal{M}}_{0,2}(X, d))$ for any $i\in\{1,2\}$.  Recall $\xi^u=\iota^*\sigma^u$ and note $PD([ev_i^{-1}(X)])=ev_i^*\sigma^{s_{n-1}}$. Hence, we have
\begin{align*}
 \langle \xi^{u}, \xi^{w} \rangle_{0,2,d}^{X} =&\int_{[\overline{\mathcal{M}}_{0,2}(X, d)]} ev_{1}^{*}\iota^*\sigma^{u}  \cup ev_{2}^{*}\iota^*\sigma^{w}\\
 =&\int_{[\overline{\mathcal{M}}_{0,2}(X, d)]} \iota^*(ev_{1}^{*}\sigma^{u}  \cup ev_{2}^{*}\sigma^{w})\\
 =&
\int_{[\iota(\overline{\mathcal{M}}_{0,2}(X, d))]}ev_{1}^{*}\sigma^{u}  \cup ev_{2}^{*}\sigma^{w} \\
 =&\int_{[\overline{\mathcal{M}}_{0,2}(F\ell_{n}, \iota_{*}d)]} ev_{1}^{*}\sigma^{u}  \cup ev_{2}^{*}\sigma^{w} \cup ev_2^{*}\sigma^{s_{n-1}}.
\qedhere
\end{align*}

 \end{proof}

\subsection{Vanishing of  Gromov-Witten invariants with $d_{n-1}\geq 2$}
\subsubsection{Vanishing by curve neighborhood technique}
 We  use the curve neighborhood technique developed by Buch and Mihalcea \cite{BuMi}, to show the vanishing of    Gromov-Witten invariants of degree $d$ with $d_{n-1}\geq 2$ and $d\neq \alpha_{in}+\alpha_{n-1, n}$.

\begin{defn} Let $d\geq 0$ and $u\in S_n$, which further satisfies $u(n)\neq 1$ for $Y=X$.
    The curve neighborhood  ${\Gamma}_{d}^{Y}(X_{u})$   of $X_{u}$ of degree $d$ is a reduced subscheme of $Y$  defined by
    $${\Gamma}_{d}^{Y}(X_{u})={ev}_2({ev}_1^{-1}(X_u)).$$
\end{defn}

 The permutation $z_d\in S_n$ associated with $d\geq 0$ is defined by using the Hecke product $\bullet$ on $S_n$ as follows. Note
$$  w \bullet s_{i}=\begin{cases}
    ws_i,&\mbox{if }\ell(ws_{i}) >  \ell(w), \\
    w,&\mbox{otherwise}.
\end{cases}$$
Take a sequence $(\alpha_{i_1j_1},\alpha_{i_2j_2},\cdots,\alpha_{i_kj_k})$ of maximal elements $\alpha_{i_rj_r}=\sum_{m=i_r}^{j_r-1}[X_{s_m}]$ with respect to $d$; that is, each $\alpha_{i_rj_r}$ is maximal in the sense $\alpha_{i_rj_r}\in A_r:=\{\alpha_{ab}\mid d-\alpha_{ab}-\sum_{m=1}^{r-1}\alpha_{i_mj_m}\geq 0\}$ with $j_r-i_r=\max\{b-a\mid \alpha_{ab}\in A_r\}$.
  Then $k$ depends only on $d$, and
  $$z_d:=t_{i_1 j_1}\bullet t_{i_2 j_2} \bullet \cdots \bullet t_{i_k j_k}\in S_n $$   is also independent of the choice  of  a sequence  of maximal elements with respect to  $d$.

\begin{prop}[\protect{\cite[Theorem 5.1]{BuMi}}] \label{BM}
  ${\Gamma}^{F\ell_n}_{d}(X_{u})=X_{u \bullet z_d}$, for    $u\in  S_{n}$ and   $d\geq 0$.
\end{prop}

\begin{lemma} \label{lemma:len_zd}
  Let $d\geq 0$ with $d_{n-1}\geq 2$. Then we have $\ell(z_d) \leq  \langle d, c_1(T_{X})\rangle-1$, with equality holding only if $d=\alpha_{in}+\alpha_{n-1,n}$ for some $i$.
\end{lemma}

\begin{proof}
Take a sequence $(\alpha_{i_1j_1},\alpha_{i_2j_2},\cdots,\alpha_{i_kj_k})$ of maximal elements with respect to $d$. We may arrange that $j_r=n$ for $1\leq r\leq d_{n-1}$ (by noting that   the corresponding transpositions of the form $t_{cn}$ are disjoint with that of the form $t_{ab}$ with $b<n$; otherwise, $b=c$, and $\alpha_{an}$ would be a bigger element than $\alpha_{cn}$).
Note $\ell(t_{ab})=2b-2a-1=2|\alpha_{ab}|-1$, $\ell(t_{an}\bullet t_{bn})\leq \ell(t_{an})+\ell(t_{bn})-1$ (since $t_{an}$ (resp. $t_{bn}$) has a reduced expression ending (resp. starting) with $s_{n-1}$). Hence, $k\geq d_{n-1}\geq 2$, and we have
\begin{align*}
   \ell(z_d) &\leq \ell(t_{i_1 n} \bullet \cdots \bullet  t_{i_{d_{n-1}} n})+\ell(t_{i_{d_{n-1}+1} j_{d_{n-1}+1} } \bullet  \cdots \bullet   t_{i_k j_k})\\
   &\leq \sum_{r=1}^{d_{n-1}}
    \ell(t_{i_r n})-(d_{n-1}-1)+ \sum_{r=d_{n-1}+1}^{k} \ell(t_{i_r j_r})\\
    &=2|d|-k-(d_{n-1}-1)\\
    &\leq 2|d|-d_{n-1}-(d_{n-1}-1)=\langle d, c_1(T_X)\rangle -d_{n-1}+1.
\end{align*}
 Hence, $\ell(z_d)\leq \langle d, c_1(T_X)\rangle
    -1$, with equality holding only if $k=d_{n-1}=2$.
    When   equality holds, all the above inequalities are equalities.
    In particular,
    $d=\alpha_{in}+\alpha_{jn}$ for some $i\leq j$, and $\ell(t_{in}\bullet (s_{n-1}t_{jn}))=\ell(t_{in}\bullet t_{jn})=\ell(t_{in})+\ell(t_{jn})-1$. Consequently $t_{in}\bullet (s_{n-1}t_{jn})=t_{in}(s_{n-1}t_{jn})$, and it has a reduced expression ending with $s_{n-1}$.
     It follows that
       $t_{in}(s_{n-1}t_{jn})(n-1)>t_{in}(s_{n-1}t_{jn})(n)$. This results in a contradiction $i>t_{in}(j)$ in the case  $j<n-1$.
\end{proof}

\begin{prop} \label{prop: d_vanish}
   Let $d\geq 0$ satisfy   $d_{n-1} \geq 2$ and $d\neq \alpha_{in}+\alpha_{n-1, n}$ for $1\leq i\leq n-1$.  Then for any $\beta, \gamma \in H^*(X)$, we have
     $\langle \beta, \gamma \rangle_{0,2,d}=0$.
\end{prop}
\begin{proof}
    Take any $u, v\in S_n$ with $u, v\leq w_0s_{n-1}$. Note $\Gamma_{d}^{X}(X_{u})\subseteq \Gamma_{d}^{F\ell_n}(X_{u})=X_{u\bullet z_d}$
 by  \cref{BM}. Since $d_{n-1}\geq 2$ and $d\neq \alpha_{in}+\alpha_{n-1, n}$,     by \cref{lemma:len_zd}, we have
    $$\dim\Gamma_{d}^{X}(X_{u})
        \leq \dim (X_{u\bullet z_d})=\ell(u\bullet z_d) \leq \ell(u)+\ell(z_d)<\ell(u)+\langle d, c_1(T_X)\rangle-1.
    $$

 Denote $PD([X_u])$ as $[X_u]$ by abuse of notation.    Using projection formula, we have
    $$\int_{[\overline{\mathcal{M}}_{0,2}(X,d)]^{\rm vir}}   ev_1^{*} ([X_u]) \cup ev_2^{*}([X_v])=\int_{X}  (ev_{2})_{*} (ev_1^{*} ([X_u]) \cap [\overline{\mathcal{M}}_{0,2}(X,d)]^{\rm vir}  )\cup  [X_v]. $$
The cycle $(ev_{2}) (ev_1^{-1} ([X_u]))$  is supported on  the curve neighborhood $\Gamma_{d}^{X}(X_{u})$, and the pushforward  $(ev_{2})_{*} (ev_1^{*} ([X_u]) \cap [\overline{\mathcal{M}}_{0,2}(X,d)]^{\rm vir}   )$ is non-zero only if the curve neighborhood $\Gamma_{d}^{X}(X_{u})$ has components of dimension
$$\text{expdim}~  \overline{\mathcal{M}}_{0,2}(X,d)-(\dim X-\ell(u))= \langle d, c_1(T_{X})\rangle-1+\ell(u). $$
 However, such components do not exist by the above estimation of $\dim \Gamma^X_d(X_u)$.

 Hence,
    $\langle [X_{u}], [X_{v}] \rangle_{0,2,d}=\int_{[\overline{\mathcal{M}}_{0,2}(X,d)]^{\rm vir}}   ev_1^{*} ([X_u]) \cup ev_2^{*}([X_v])=0$.
    Since $\{[X_u]\}_{u\leq w_0s_{n-1}}$ is a basis of $H^*(X, \bbZ)$, the statement follows.
\end{proof}

\subsubsection{Vanishing for specific degrees} It remains to show the vanishing of Gromov-Witten invariants for  $d=\alpha_{in}+\alpha_{n-1, n}$. \footnote{
Special treatments for Gromov-Witten invariants of specific degrees were also required in certain blowups of Grassmannians in
\cite[Lemma 3.7]{HKLS24} and \cite[Theorem 4.1 (b)]{HKLS}. }

Let us consider
$$\mathcal{P}=F\ell_{1,\ldots,n-2;n},\qquad
Y= \{V_\bullet\in\mathcal{P}\mid F_1\subseteq V_{n-2}\}.$$
We have a natural projection $\pi:X\to \mathcal{P}$ by forgetting $V_{n-1}$.
The fiber over a point of $Y$ is $\mathbb{P}^1$, while the fiber over a point of $\mathcal{P}\setminus Y$ is a point.
Note
$$ \dim \mathcal{P}=\dim X,\qquad
\dim Y=\dim \mathcal{P}-2.$$

\begin{prop} \label{prop: d_specific}
   Let $d=\alpha_{in}+\alpha_{n-1, n}$ where  $1\leq i\leq n-1$.  Then for any $\beta_1, \beta_2 \in H^*(X)$, we have
     $\langle \beta_1, \beta_2 \rangle_{0,2,d}=0$.
\end{prop}

\begin{proof} We discuss the degrees in two cases.
\medbreak
\paragraph{\bf Case $i=n-1$}
In this case,  $\pi_*d = 0$. Therefore any stable map of degree $d$ is contained in a fiber of $\pi$ at some point of $Y$.
This defines a morphism from the moduli space to $\mathcal{P}$ whose image is contained in $Y$.
Denoting ${\rm ev}=ev_1\times ev_2$ and   by $\Delta$ the diagonal map, we have
$$\xymatrix{
\overline{\mathcal{M}}_{0, 2}(X,d)\ar[r]^-{{\rm ev}}\ar[d] & X\times X\ar[d]\\
\mathcal{P}\ar[r]^-{\Delta} & \mathcal{P}\times \mathcal{P}}$$
The space $\overline{\mathcal{M}}_{0,2}(X, d)$ has expected dimension
$$\dim X+\deg_X q_{n-1}^2+2-3=\dim X+1.$$
The image of ${\rm ev}$ lies in the preimage of $Y$ of $X\times_{\mathcal{P}} X$, which is a $\mathbb{P}^1\times \mathbb{P}^1$-bundle over $Y$.
Thus its dimension is
$$\dim Y+2=\dim X-2+2=\dim X.$$
Thus
${\rm ev}_*[\overline{\mathcal{M}}_{0,2}(X,d)]^{\rm vir}=0$.

\medbreak
\paragraph{\bf Case $i<n-1$}
We have a similar commutative diagram
$$\xymatrix{
\overline{\mathcal{M}}_{0,2}(X, d)\ar[r]^-{{\rm ev}}\ar[d]_{\hat{\pi}} & X\times X\ar[d]^{\pi\times\pi}\\
\overline{\mathcal{M}}_{0, 2}(\mathcal{P},\pi_*{d})\ar[r]^-{{\rm ev}_{\mathcal{P}}} & \mathcal{P}\times \mathcal{P}.}$$
Note that $\deg_X \alpha_{j, j+1}=2=\deg_\mathcal{P}\pi_*\alpha_{j, j+1}$ for $j<n-2$, \quad $\deg_X \alpha_{n-2, n-1}=2$ $=\deg_\mathcal{P}\pi_*\alpha_{n-2, n-1}-1$ and  $\deg_X \alpha_{n-1, n}=1$. Hence, we have
\begin{align*}
\exp\dim \overline{\mathcal{M}}_{0,2}(X, d) & = \dim X+\deg_X q^d+2-3,\\
\exp\dim \overline{\mathcal{M}}_{0, 2}(\mathcal{P},\pi_*{d}) & =
\dim \overline{\mathcal{M}}_{0, 2}(\mathcal{P},\pi_*{d})  =
\exp\dim \overline{\mathcal{M}}_{0,2}(X, d)-1.
\end{align*}
Moreover the map $\hat{\pi}$ cannot be surjective, since any stable map in the image has an extra constrain that $C$ meets $Y$.
Combining with the known fact that $\overline{\mathcal{M}}_{0, 2}(\mathcal{P}, \pi_*d)$ is irreducible, this reduces one more dimension:
$$\dim(\operatorname{im}(\hat{\pi}))\leq \dim \overline{\mathcal{M}}_{0, 2}(\mathcal{P},\pi_*{d})-1.$$

As a result,
${\rm ev}_*[\overline{\mathcal{M}}_{0,2}(X, d)]^{\rm vir}$ is supported over
$$(\pi\times \pi)^{-1}(\operatorname{im}({\rm ev}_{\mathcal{P}}\circ \hat{\pi})).$$
Since the fibers of $\pi\times \pi$ have dimension at most $2$, it has dimension at most
\begin{align*}
\dim(\operatorname{im}({\rm ev}_{\mathcal{P}}\circ \hat{\pi}))+2
& \leq
\dim(\operatorname{im}\hat{\pi})+2
\\
& \leq
\dim \overline{\mathcal{M}}_{0,2}(\mathcal{P},\pi_*{d})-1+2
={\rm exp}\dim \overline{\mathcal{M}}_{0,2}(X, d).
\end{align*}
If ${\rm ev}_*[\overline{\mathcal{M}}_{0,2}(X, d)]^{\rm vir}\neq 0$, then equality must hold throughout.
{Since  $\pi\times \pi$ has two-dimensional fibers only over points of $Y\times Y$}, equality can hold only if
${\rm ev}_{\mathcal{P}}$ restricts to a morphism $Z_1\to   Z_2$ that is generically finite, where $Z_1$ is  a component of $\operatorname{im}\hat{\pi}$ of codimension $1$ in $\overline{\mathcal{M}}_{0, 2}(\mathcal{P},\pi_*{d})$
and $Z_2$ is a component of $\operatorname{im}({\rm ev}_{\mathcal{P}}\circ \hat \pi)$ contained in $Y\times Y$. By (the proof of) \cite[Lemma 3.8]{BCMP}, $Z_2$ is contained in a locally trivial fibration over   $Y_1\subset Y$ with fiber $\Gamma^{\mathcal{P}}_{\pi_*d}(y)$, where $Y_1$ denotes the natural projection of $Z_2\subset Y\times Y$ to the first factor. By \cite[Theorem 5.1]{BuMi}, the fiber $\Gamma^{\mathcal{P}}_{\pi_*d}(y)$
is a Schubert variety of $\mathcal{P}$, indexed by the permutation $z_{\pi_*d}^{\mathcal{P}}=z_{\alpha_{in}}s_{n-1}=t_{in}s_{n-1}$. Thus it is of dimension $\ell(z_{\pi_*d}^{\mathcal{P}})=\ell(t_{in})-1$.
\begin{align*}
    \dim Z_2&\leq  \dim Y_1+\dim \Gamma_{\pi_*d}^{\mathcal{P}}(y)
\leq \dim Y+\ell(t_{in})-1=\dim X-2+2(n-i)-1-1,\\
 \dim Z_1&=\dim \overline{\mathcal{M}}_{0,2}(\mathcal{P},\pi_*d) -1=  \exp\dim \overline{\mathcal{M}}_{0,2}(X, d)-2=\dim X+2(n-i)-1-2,
\end{align*}
 resulting in  a contradiction $0=\dim Z_2-\dim Z_1\leq -1$.
\end{proof}

\subsection{Proof of \cref{thm: QCFbody}}
 Note that $\{PD([X_u])\}_u$ is the dual basis of $\{\xi^u\}_u$ with respect to the Poincar\'e pairing.
 Write $PD([X_u])=\sum_{\gamma}a_{\gamma}^{u} \xi^{\gamma}$.
Using the projection formula, $$\delta_{u, w}=\int_{[X]} PD([X_u]) \cup \xi^{w}=\int_{[X]} \iota^*(\sum_{\gamma}a_{\gamma}^{u} \sigma^{\gamma})  \cup \iota^*\sigma^{w}=\int_{[F\ell_{n}]} \sum_{\gamma}a_{\gamma}^{u} \sigma^{\gamma} \cup \iota_*(\iota^*\sigma^{w}).$$
Note $\iota_*(\iota^*\sigma^{w})
=\iota_*(\iota^*\sigma^{w}\cup \xi^{\rm id})=\sigma^w\cup PD([X])=\sigma^{w}\cup \sigma^{s_{n-1}}$.
The permutation $u$ varies in $S_n$ with $u(n)\neq 1$. Hence, $$\sum_{\gamma}a_{\gamma}^{w} \sigma^{\gamma} \cup \sigma^{s_{n-1}} =(\sigma^{w})^{\vee}+\sum_{\eta(n)=1}b_\eta(\sigma^\eta)^\vee.$$
 We have
\begin{align*}
    \xi^{s_r} \star \xi^{u}&=\sum_{w,d} \langle \xi^{s_r}, \xi^{u}, PD([X_w]) \rangle_{0,3,d}^{X}  \xi^{w} q^{d}  \\
    &=\sum_{w \leq w_0s_{n-1}} \langle \xi^{s_r}, \xi^{u}, PD([X_w]) \rangle_{0,3,0}^{X} \xi^{w}+ \sum_{ w\leq w_0s_{n-1},d \neq 0 } \langle \xi^{s_r}, \xi^{u}, PD([X_w]) \rangle_{0,3,d}^{X} \xi^{w} q^{d}.
\end{align*}
By \cref{prop: d_vanish} and \cref{prop: d_specific}, there are no $q^d$-terms in the second sum whenever $d_{n-1}\geq 2$. By the divisor axiom in Gromov-Witten theory, for $d\neq 0$ we have
$\langle \xi^{s_r}, \xi^{u}, PD([X_v]) \rangle_{0,3,d}^{X}=\int_{d} \xi^{s_{r}} \langle \xi^{u}, \sum_{\gamma}a_{\gamma}^{v} \xi^{\gamma} \rangle_{0,2,d}$.

For $d_{n-1}=1$, by \cref{lowdg GW},  we have
\begin{align*}
         &\int_{d} \xi^{s_{r}} \langle \xi^{u}, \sum_{\gamma}a_{\gamma}^{w} \xi^{\gamma} \rangle_{0,2,d}^{X} q^{d}  \\
         =&\int_{d} \sigma^{s_{r}} \langle \sigma^{u}\cup \sigma^{s_{n-1}}, \sum_{\gamma}a_{\gamma}^{w}\sigma^{\gamma} \cup \sigma^{s_{n-1}} \rangle_{0,2,d}^{F\ell_n}q^{d}- \int_{d} \xi^{s_{r}}\delta_{q^{d},q_{n-1}}\langle  \sigma^{u} , \sum_{\gamma}a_{\gamma}^{w}\sigma^{\gamma} , \sigma^{s_{n-1}} \rangle_{0,3,0}^{F\ell_n}q_{n-1} \\
         =&\int_{d} \sigma^{s_{r}} \langle \sigma^{u}\cup \sigma^{s_{n-1}}, \sum_{\gamma}a_{\gamma}^{w}\sigma^{\gamma} \cup \sigma^{s_{n-1}} \rangle_{0,2,d}^{F\ell_n}q^{d}- \int_{d} \xi^{s_{r}}\delta_{q^{d},q_{n-1}} q_{n-1}\int_{[F\ell_n]} \sigma^{u}\cup (\sum_{\gamma}a_{\gamma}^{w}\sigma^{\gamma} )\cup \sigma^{s_{n-1}} \\
          =&\langle \sigma^{s_r}, \sigma^{u}\cup \sigma^{s_{n-1}}, (\sigma^{w})^{\vee} +\sum_{\eta(n)=1}b_\eta(\sigma^\eta)^\vee\rangle_{0,3,d}^{F\ell_n}q^{d}- \int_{d} \xi^{s_{r}}\delta_{q^{d},q_{n-1}} q_{n-1}\int_{[F\ell_n]} \sigma^{u}\cup ((\sigma^{w})^{\vee} +\sum_{\eta(n)=1}b_\eta(\sigma^\eta)^\vee)\\
           =&\langle \sigma^{s_r}, \sigma^{u}\cup \sigma^{s_{n-1}}, (\sigma^{w})^{\vee} \rangle_{0,3,d}^{F\ell_n}q^{d}- \int_{d} \xi^{s_{r}} \delta_{q^{d},q_{n-1}} q_{n-1}\delta_{u, w}
     \end{align*}
The last equality holds by noting
$\int_{[F\ell_n]}\sigma^u\cup (\sigma^{\eta})^\vee=0$ (since $u(n)\neq 1)$ and
$$\langle \sigma^{s_r}, \sigma^{u}\cup \sigma^{s_{n-1}},  (\sigma^{\eta})^\vee\rangle_{0,3,d}^{F\ell_n}=\sum_{\tilde u}\langle \sigma^{s_r}, \sigma^{\tilde u},  (\sigma^{\eta})^\vee\rangle_{0,3,d}^{F\ell_n}=\sum_{\tilde u} N_{s_r, \tilde u}^{\eta, d}=0$$
 for any permutation $\eta$ with $\eta(n)=1$. Indeed, by the quantum Monk's formula for $F\ell_n$, $u\lessdot_{n-1} \tilde u$ and for $\hat w\in S_n$,
 $N_{s_r, \tilde u}^{\hat w, d}\neq 0$ only if $\tilde u\lessdot^q \hat w=\tilde u t_{an}$ for some $a$ (since $d_{n-1}=1$).
 \begin{enumerate}
     \item[i)] If $\tilde u\leq w_0s_{n-1}$, then $\hat w\leq \tilde u\leq w_0s_{n-1}$, i.e. $\hat w(n)\neq 1$;
     \item[ii)] If   $\tilde u\not \leq w_0s_{n-1}$, i.e. $\tilde u(n)=1$, then $\hat w(n)=\tilde ut_{an}(n)=\tilde u(a)\neq \tilde u(n)=1$.
 \end{enumerate}
Thus the sum  vanishes for any $\eta$ with $\eta(n)=1$.

 For $d \neq 0$ with $d_{n-1}=0$, by \cref{lowdg00}, we have
 \begin{align*}
         \int_{d} \xi^{s_{r}} \langle \xi^{u}, \sum_{\gamma}a_{\gamma}^{w} \xi^{\gamma} \rangle_{0,2,d}^{X} q^{d}
         &=\int_{d} \sigma^{s_{r}} \langle \sigma^{u}, \sum_{\gamma}a_{\gamma}^{w}\sigma^{\gamma} \cup \sigma^{s_{n-1}} \rangle_{0,2,d}^{F\ell_n}q^{d} \\
         &=\langle \sigma^{s_r}, \sigma^{u}, (\sigma^{w})^{\vee}+\sum_{\eta(n)=1}b_\eta (\sigma^\eta)^\vee \rangle_{0,3,d}^{F\ell_n}q^{d}\\
       & =\langle \sigma^{s_r}, \sigma^{u}, (\sigma^{w})^{\vee}\rangle_{0,3,d}^{F\ell_n}q^{d}.
     \end{align*}
The last equality holds again by noting
$\langle \sigma^{s_r}, \sigma^{u}, (\sigma^{\hat w})^{\vee}\rangle_{0,3,d}^{F\ell_n}q^{d}=0$ unless $u\lessdot_r^{q} \hat w$, implying $\hat w(n)\neq 1$.
Hence, we are done. \hfill $\Box$

\begin{remark}\label{lengthQCF}
   The arguments  i) and ii) in the above proof say that
   for any $u, w\in S_n$ and any $a, k$, the hypothesis $u\lessdot_{n-1} wt_{an} \lessdot_k^q w$ implies $w(n)\neq 1$, i.e. $w\leq w_0s_{n-1}$.
\end{remark}
 \section{Quantum Schubert polynomials for  $X$}
This section is devoted to a proof of \cref{thm: QSP}, namely for any $w$, the quantum Schubert polynomial $\mathfrak{S}_w^q$ of Fomin, Gelfand and Postnikov represents the pullback Schubert class $\xi^w$, under the canonical ring isomorphism in \cref{thm: QHX}.

Recall $$q_{ab}:= q_aq_{a+1}\cdots q_{b-1}\quad\mbox{for}\quad  1\leq a<b\leq n.$$
  \subsection{Quantum Schubert polynomials}

The classical Schubert polynomials were introduced by
   Lascoux and  Sch\"utzenberger
  \cite{LascouxSchu} by using the divided difference operators $\partial_i$  of Bernstein, Gelfand and Gelfand \cite{BGG}.
Precisely, for $f=f(x_1, \cdots, x_n)\in \bbZ[x]$ and $w\in S_n$,   denote $wf=f(x_{w^{-1}(1)}, \cdots, x_{w^{-1}(n)})$. Then $\partial_if:={f-s_if\over x_i-x_{i+1}}\in \bbZ[x]$, and the classical Schubert polynomials $\mathfrak{S}_w(x)$ are recursively defined by
\begin{equation}
    \mathfrak{S}_{w_0}=x^{n-1}_1x_2^{n-2}\cdots x_{n-1}\qquad \mbox{and}\quad \mathfrak{S}_{ws_i}=\partial_i\mathfrak{S}_w\quad \mbox{whenever} \quad \ell(ws_i)=\ell(w)-1.
\end{equation}
The following were shown in \cite{LascouxSchu}.
\begin{enumerate}
    \item $\Phi(\sigma^w)=[\mathfrak{S}_w(x)]$ under the canonical ring isomorphism $\Phi$ in \cref{prop:HFl}.
    \item $\{e^1_{i_1}e^2_{i_2}\cdots e^{n-1}_{i_{n-1}}\}_{0\leq i_j\leq j}$ form a $\bbZ$-basis of $\bbZ[x]$.
\end{enumerate}
Therefore, we have the linear expansion
$\mathfrak{S}_w=\sum \alpha_{i_1\ldots i_{n-1}}e^1_{i_1}e^2_{i_2}\cdots e^{n-1}_{i_{n-1}}$.
In \cite{FGP}, Fomin, Gelfand and Postnikov   introduced the quantum Schubert polynomial
\begin{equation}
    \mathfrak{S}_w^q:= \sum \alpha_{i_1\ldots i_{n-1}}E^1_{i_1}E^2_{i_2}\cdots E^{n-1}_{i_{n-1}}.
\end{equation}
They also showed $\Phi_q(\sigma^w)=[\mathfrak{S}_w^q]$,   under the   canonical ring isomorphism   \cite{GiKi, CF, Kim}
\begin{equation}\label{QHFlpresentation}
    \Phi_q: QH^*(F\ell_n, \bbZ)\longrightarrow\mathbb{Z}[x_1, \cdots, x_n, q_1, \cdots, q_{n-1}]/(E_1^n, \cdots, E_n^n).
\end{equation}

Recall the (quantum) $r$-Bruhat order $\lessdot_r$ (resp. $\lessdot_r^q$) defined in \cref{kBruhat} (resp.  \cref{quantumkBruhat}).
The following is the quantum Monk's formula on the level of polynomials, proved in the first half of \cite[Theorem 7.1]{FGP}.
\begin{prop}[Quantum Monk's formula]\label{prop: quantumMonkFlalg}
For $u\in S_n$ and $1\leq r<n$, in $\mathbb{Z}[x]$ we have
$$\mathfrak{S}_{s_r}^q\mathfrak{S}_{u}^q=(x_1+\cdots+x_r)\mathfrak{S}_{u}^q=\sum_{u\lessdot_rut_{ab}}\mathfrak{S}_{ut_{ab}}^q+\sum_{u\lessdot_r^q ut_{ck}}q_{ck}\mathfrak{S}_{ut_{ck}}^q.$$
\end{prop}
\begin{lemma}\label{Bruhuatcoverlem}
    Let $u,w\in S_n$ and $1\leq a<b\leq n$.
    \begin{enumerate}
        \item  $u\lessdot ut_{ab}$ if and only if
     $u(a)<u(b)$ and for any $a<c<b$, we have $u(c)\notin [u(a), u(b)]$.
     \item $u\lessdot^q ut_{ab}$  if and only if
     $u(a)>u(b)$ and for any $a<c<b$, we have $u(c)\in [u(b), u(a)]$.
    \end{enumerate}
\end{lemma}
\begin{proof}
  Note $\ell(t_{ab})=2b-2a-1$. The statement follows from a direct counting of the number of inversions, which defines the length of a permutation.
\end{proof}
The next proposition is a special case of the second half of  \cite[Theorem 7.1]{FGP}, with a slightly more precise description than that in loc. cit.; see also \cite[Theorem 4]{LOTRZ}. This special case will play a crucial role in our proof of \cref{thm: QSP}. A permutation
$w\in S_n$ is said to have a descent  at the $k$-th position if $w(k+1)<w(k)$.
\begin{prop}[Transition equation]\label{InductiveQSchu}
     Let $w\in S_n\setminus\{  {\rm id}\}$. Denote by $i$ the last descent position of $w$. Take the maximal $j$ with $w(j)<w(i)$.
    Then  $u:=wt_{ij}$ satisfies $u\lessdot w$, and we have
    \begin{align*}
    \mathfrak{S}^q_w & =
    x_i\mathfrak{S}^q_{u}
    +\sum_{u\lessdot ut_{hi}} \mathfrak{S}^q_{ut_{hi}}
     + \sum_{u\lessdot^q ut_{hi}} q_{hi}
    \mathfrak{S}^q_{ut_{hi}}
    - \sum_{u\lessdot^q ut_{ik}} q_{ik}
    \mathfrak{S}^q_{ut_{ik}}.
    \end{align*}
\end{prop}
\begin{proof} It follows directly from the definition that $i<j$ and that for any $i<c<j$, $u(c)=w(c)<w(j)=u(i)$ (where the inequality holds since $i$ is the last descent position). Thus $u\lessdot w$ by \cref{Bruhuatcoverlem} (1).

   By \cref{prop: quantumMonkFlalg},   we compare the two quantum Monk's formulas
    \begin{align*}
        (x_1+x_2+\cdots +x_{i-1})\mathfrak{S}_u^q&=\sum_{u\lessdot_{i-1}ut_{ab}}\mathfrak{S}_{ut_{ab}}^q+\sum_{u\lessdot_{i-1}^q ut_{ck}}q_{ck}\mathfrak{S}_{ut_{ck}}^q;\\
        (x_1+x_2+\cdots +x_{i})\mathfrak{S}_u^q&=\sum_{u\lessdot_iut_{ab}}\mathfrak{S}_{ut_{ab}}^q+\sum_{u\lessdot_i^q ut_{ck}}q_{ck}\mathfrak{S}_{ut_{ck}}^q.
    \end{align*}

   Notice that if $c\leq i-1<i<k$, then  $u\lessdot_{i-1}^q ut_{ck}$  implies $u\lessdot_{i}^q ut_{ck}$.
   Therefore the difference between the quantum parts of  the two products  occurs exactly when either $k=i$ in the first product or $c=i$ in the second product.
  Hence, the quantum part of the statement follows.

  For the classical part, the same argument applies once we show that $u\lessdot ut_{ib}$ implies $b=j$. Indeed, assume
  $u\lessdot ut_{ib}$ for some $b\neq j$. Then  $w(j)=u(i)<u(b)=w(b)$ by \cref{Bruhuatcoverlem}.
   Since $i$ is the last descent position of $w$, it follows that $j<b$.
   Furthermore, we have $u(j)=w(i)<w(b)=u(b)$, since $j$ is  maximal with respect to $w(j)<w(i)$.
   But then we would have $i<j<b$ and $u(i)<u(j)<u(b)$, contradicting   $u\lessdot ut_{ib}$ by \cref{Bruhuatcoverlem}.
\end{proof}

\begin{lemma}\label{InductiveCor}
   Let $w\in S_n\setminus\{  {\rm id}\}$ with $w(n)\neq 1$. Then all the permutations $v$ occurring on the right hand side of the formula of $\mathfrak{S}_w^q$ in  \cref{InductiveQSchu} satisfy $v(n)\neq 1$.
\end{lemma}
\begin{proof}
 With the same notation as in \cref{InductiveQSchu}, if $v=u=wt_{ij}$, then $u(n)\geq u(j)>u(i)\geq 1$. If $v=ut_{hi}$, then $v(n)=u(n)\neq 1$ by noting $i<n$. It remains to discuss the case $v=ut_{ik}$.
Since $u\lessdot^q ut_{ik}$, by \cref{Bruhuatcoverlem}, $u(k)<u(i)=w(j)$, so $k<j$ since $i$ is the last descent position. In particular, $k\neq n$. Thus, $v(n)=u(n)\neq 1$.
\end{proof}

\begin{defn}\label{linearorder}
   For $u, w\in S_n$, we say $u\prec w$ if and only if
  $$(\ell(u),-u(n), -u(n-1),\cdots, -u(1))<(\ell(w),-w(n), -w(n-1),\cdots, -w(1))$$
  with respect to the  lexicographic order.  This defines a total order $\prec$ on $S_n$.
\end{defn}
\begin{lemma}\label{orderinInductiveformula}
    Let $w\in S_n\setminus\{  {\rm id}\}$. Then all the permutations occurring on the right hand side of the formula of $\mathfrak{S}_w^q$ in  \cref{InductiveQSchu} are strictly smaller than $w$ with the order $\prec$.
\end{lemma}
\begin{proof}
    Let $u$ be as in  \cref{InductiveQSchu}. Then $u\prec w$ since $\ell(u)=\ell(w)-1$. For $u\lessdot ut_{hi}$,  we have $\ell(w)=\ell(ut_{hi})=\ell(u)+1$. By \cref{Bruhuatcoverlem}, we have $w(j)=u(i)<u(j)=ut_{hi}(j)$. Combining this  with the property $w(a)=u(a)=ut_{hi}(a)$ for $a>j$, we obtain $ut_{hi}\prec w$. Permutations in the quantum part are all of length   smaller than $\ell(w)$, and hence are strictly   smaller than $w$ with respect to the total order $\prec$.
\end{proof}

\subsection{Proof of \cref{thm: QSP}}
To achieve our aim, we first show that the pullback Schubert classes $\xi^w$ admit exactly the same transition equations in the quantum cohomology $QH^*(X)$ as that for $\mathfrak{S}_w^q$ on the level of polynomials.
\begin{prop}\label{GeometricTransitionformula}
   With the same notation  as in \cref{InductiveQSchu}, in    $QH^*(X)$ we have
    \begin{align*}
    \xi^w & =
    x_i\xi^{u}
   +\sum_{u\lessdot ut_{hi}} \xi^{ut_{hi}}
     + \sum_{u\lessdot^q ut_{hi}} q_{hi}
    \xi^{ut_{hi}}
    - \sum_{u\lessdot^q ut_{ik}} q_{ik}
    \xi^{ut_{ik}}.
    \end{align*}
\end{prop}
\begin{proof}
 Recall $w=ut_{ij}$ and that $i$ is the last descent of $w$, so  $i\leq n-1$ and $x_i=\xi^{s_i}-\xi^{s_{i-1}}$, where $\xi^{s_0}:=0$.  We compare the two quantum Monk-Chevalley formulas

\begin{align*}
    \xi^{s_{i-1}}\star \xi^u&= \!\!\! \!\sum_{u\lessdot_{i-1} ut_{ab}\leq w_0s_{n-1}}  \!\!\!\!\!\!\!\!\!\xi^{ut_{ab}}+ \!\!\!\!\!   \sum_{u\lessdot_{i-1}^q ut_{ck}\,\,{\rm with}\,\, k<n}  \!\!\!\!\!\!\!\!\!\xi^{ut_{ck}}q_{ck}+    \!\!\!\!\!\!\!\!\!\sum_{u\lessdot  ut_{an}\lessdot_{i-1}^q ut_{an}t_{cn} \leq w_0s_{n-1}}  \!\!\!\!\!\!\!\!\!\!\!\!\xi^{ut_{an}t_{cn}}q_{cn}-0,\\
     \xi^{s_{i}}\star \xi^u&=\sum_{u\lessdot_{i} ut_{ab}\leq w_0s_{n-1}} \!\!\!\!\!\!\!\!\!\xi^{ut_{ab}}+    \sum_{u\lessdot_{i}^q ut_{ck}\,\,{\rm with}\,\, k<n} \!\!\!\!\!\!\!\!\!\xi^{ut_{ck}}q_{ck}+   \!\!\! \!\!\! \sum_{u\lessdot  ut_{an}\lessdot_{i}^q ut_{an}t_{cn} \leq w_0s_{n-1}} \!\!\!\!\!\!\!\!\!\xi^{ut_{an}t_{cn}}q_{cn}-\delta_{i,n-1}q_{n-1}\xi^u.\\
\end{align*}

    The classical part follows from the same argument as in \cref{InductiveQSchu}, where the constraint $ut_{hi}\leq w_0s_{n-1}$ is redundant by \cref{InductiveCor}.
      Again note that if $c\leq i-1<i<k$, then  $u\lessdot_{i-1}^q ut_{ck}$  implies $u\lessdot_{i}^q ut_{ck}$.
   Therefore the difference between the   quantum parts  of the two quantum products  occurs exactly when either $k=i$ in the first product or $c=i$ in the second product.
The   $k=i$ part in the first product is exactly the second sum in the equation for $\xi^w$ in the statement.
  It remains to show that the rest is given by the $c=i$ part of $\xi^{s_i}\star \xi^u$   together with  $\delta_{i, n-1}q_{n-1}\xi^u$, namely to show
 \begin{equation}\label{eqn_lastpart}
     \sum_{u\lessdot^q ut_{ik}}
    \xi^{ut_{ik}}q_{ik}=     \sum_{u\lessdot^q ut_{ik}\,\,{\rm with}\,\, k<n} \!\!\!\!\!\!\!\!\!\xi^{ut_{ik}}q_{ik}+   \!\!\! \!\!\! \sum_{u\lessdot  ut_{an}\lessdot_{i}^q ut_{an}t_{in} \leq w_0s_{n-1}} \!\!\!\!\!\!\!\!\!\xi^{ut_{an}t_{in}}q_{in}-\delta_{i,n-1}q_{n-1}\xi^u.
 \end{equation}
 Denote by RHS (resp. LHS) the right (resp. left) hand side of the above equation to prove.
    \begin{enumerate}
        \item Case $i=n-1$. Then  $j=n$, the sum on the LHS is empty (otherwise we would have  $u\lessdot^q ut_{ik}=ut_{n-1, n}=w$, contradicting $u\lessdot w$),   and the first sum on the  RHS is empty as well.
     The constraints $$u\lessdot ut_{an}\lessdot^q ut_{an}t_{in}$$ imply $a=n-1$ by   \cref{Bruhuatcoverlem}. (Otherwise, $a<n-1=i$, then $ut_{an}(i)=u(i)<u(a)=ut_{an}(n)$, a contradiction.) Then we have $ut_{an}t_{in}=u$, which automatically satisfies $u\leq w_0s_{n-1}$ by \cref{InductiveCor}.
     Hence, $\mbox{RHS}=0+\xi^uq_{n-1}-\xi^uq_{n-1}=0=\mbox{LHS}$.
        \item Case $i<n-1$. Then $\delta_{i,n-1}=0$.

        For  $u\lessdot^qut_{ik}$ on the $\mbox{LHS}$, we have $k\neq j$ since $u\lessdot ut_{ij}=w$.
        By \cref{Bruhuatcoverlem},  $w(k)=u(k)<u(i)=w(j)$. Since $i$ is the last descent, we have $k<j$. In particular, $k<n$, thus the $\mbox{LHS}$ is equal to the first sum of the $\mbox{RHS}$. It remains to show that the second sum on the $\mbox{RHS}$ is zero. Suppose we have $u\lessdot ut_{an}\lessdot_i^qut_{an}t_{in}$. By the choice of $i,j$, $u$ also has no descent after $i$. We also note that $j$ is the minimal integer greater than $i$ that satisfies $u(i)<u(j)$. These two properties will be used over and over again in the following argument.
        If $i<a<n-1$, then $u(n-1)\in [u(a),u(n)]$, contradicting \cref{Bruhuatcoverlem}. Therefore,
        either $a=n-1$ or $a\leq i$. If $a=n-1$, by \cref{Bruhuatcoverlem}, we have $u(n-1)=ut_{an}(n)<ut_{an}(i)=u(i)$, so $j=n$. Then $ut_{an}(n-1)=u(n)=u(j)>u(i)=ut_{an}(i)$. Then $ut_{an}(n-1)\notin [ut_{an}(n),ut_{an}(i)]$, contradicting \cref{Bruhuatcoverlem}. If $a= i$, then from $u\lessdot ut_{an}$ and  \cref{Bruhuatcoverlem}, we have $j=n$. Then $ut_{an}(n-1)=u(n-1)<u(i)=ut_{an}(n)$, so $ut_{an}(n-1)\notin [ut_{an}(n),ut_{an}(i)]$, contradicting \cref{Bruhuatcoverlem}.  If $a<i$, then by \cref{Bruhuatcoverlem}, we have $ut_{an}(n-1)=u(n-1)<u(a)=ut_{an}(n)$, so $ut_{an}(n-1)\notin [ut_{an}(n),ut_{an}(i)]$, contradicting \cref{Bruhuatcoverlem}.
        \qedhere

    \end{enumerate}
\end{proof}

 \bigskip

\begin{proof}[\textbf{Proof of  \cref{thm: QSP}}]
By (the proof of) \cref{thm: QHX}, the canonical ring isomorphism
\begin{equation*}\label{QHXpresentation}
      \Psi_q:  QH^*(X, \bbZ)\longrightarrow\mathbb{Z}[x_1, \cdots, x_n, q_1, \cdots, q_{n-1}]\left/\big(\hat E_1^n, \cdots, \hat E^n_{n-1}, E_{n-1}^{n-1}\big).\right.
    \end{equation*}
   satisfies $\Psi_q(\xi^{s_i})=[x_1+x_2+\cdots+x_i]$ for $1\leq i\leq n-1$.
   Thus $\Psi_q(\xi^{w})=[\mathfrak{S}_w^q]$ holds for $w={\rm id}$ and for   $w\in S_n$ with $\ell(w)=1$ (all of which satisfy $w\leq w_0s_{n-1}$).

  By \eqref{wnneq1},  $w\leq w_0s_{n-1}$ holds if and only if $w(n)\neq 1$. By \cref{GeometricTransitionformula},  \cref{InductiveCor} and \cref{orderinInductiveformula}, every $\xi^w$ with $w(n)\neq 1$ can be written
as a $\bbZ[q]$-linear combination of classes $\xi^v$ with $v\prec w$ and $v(n)\neq 1$. By  \cref{InductiveQSchu}, $\mathfrak{S}_w^q$ can also be written as exactly the same $\bbZ[q]$-linear combination of   $\mathfrak{S}_v^q$ on the level of polynomials. Hence, the statement follows immediately from    induction on the totally ordered subset $(\{w\}_{w\leq w_0s_{n-1}}, \prec)$.
\end{proof}

\end{document}